\documentclass[a4paper,12pt, reqno]{amsart}
\usepackage{amsfonts, color}
\usepackage[textwidth=460pt]{geometry}
\usepackage{amstext, amsthm, amssymb, amsmath}
\usepackage{mathrsfs}
\usepackage{graphicx}
\usepackage{color}
\usepackage[linesnumbered,lined,ruled,commentsnumbered]{algorithm2e}

\usepackage{hyperref}
\hypersetup{colorlinks = true, citecolor = blue}

\theoremstyle{plain}
\begingroup
\newtheorem{theorem}{Theorem}[section]
\newtheorem{lemma}[theorem]{Lemma}
\newtheorem{proposition}[theorem]{Proposition}
\newtheorem{corollary}[theorem]{Corollary}
\endgroup

\theoremstyle{definition}
\begingroup
\newtheorem{defn}{Definition}

\newtheorem{assumption}{Assumption}
\endgroup

\theoremstyle{remark}
\newtheorem{remark}{Remark}

\numberwithin{equation}{section}

\newcommand{\J}{\mathcal{J}}
\newcommand{\R}{\mathbb{R}}

\newcommand{\F}{\mathcal{F}}
\newcommand{\U}{\mathcal{U}}

\newcommand{\e}{\varepsilon}
\newcommand{\weak}{\rightharpoonup}

\newcommand{\NN}{\mathbb{N}}
\newcommand{\xth}{x^\theta}
\newcommand{\X}{\mathcal{X}}
\newcommand{\Rext}{\mathbb{R}\cup\{+\infty\}}
\newcommand{\Next}{\NN\cup\{+\infty\}}

\begin{document}
\title[Minimax problems for ensembles of control-affine systems]
{Minimax problems for ensembles of control-affine systems}
\author[A. Scagliotti]{Alessandro Scagliotti}

\begin{abstract}
In this paper, we consider ensembles of control-affine systems in $\mathbb{R}^d$, and we study simultaneous optimal control problems related to the worst-case minimization.
After proving that such problems admit solutions, denoting with $(\Theta^N)_N$ a sequence of compact sets that parametrize the ensembles of systems, we first show that the corresponding minimax optimal control problems are $\Gamma$-convergent whenever $(\Theta^N)_N$ has a limit with respect to the Hausdorff distance.
Besides its independent interest, the previous result plays a crucial role for establishing the Pontryagin Maximum Principle (PMP) when the ensemble is parametrized by a set $\Theta$ consisting of infinitely many points.
Namely, we first approximate $\Theta$ by finite and increasing-in-size sets $(\Theta^N)_N$ for which the PMP is known, and then we derive the PMP for the $\Gamma$-limiting problem.
The same strategy can be pursued in applications, where we can reduce infinite ensembles to finite ones to compute the minimizers numerically. 
We bring as a numerical example the Schr\"odinger equation for a qubit with uncertain resonance frequency.
\end{abstract}

\maketitle

In this paper, we consider ensembles of control-affine systems in $\R^d$ of the form:
\begin{equation}\label{eq:intro_ens}
\begin{cases}
\dot x(t) = F_0^\theta(x(t)) + F^\theta(x(t))u(t)&
\mbox{a.e. in }[0,T],\\
x(0)=x_0^\theta,
\end{cases}
\end{equation}
where $F_0^\theta:\R ^d\to\R^d$ and $F^\theta:\R ^d\to\R^{d\times k}$ are Lipschitz-continuous mappings that define the controlled vector fields, and $u\in L^p([0,T],\R^k)$ (with $1<p<\infty$) represents the control.
In \eqref{eq:intro_ens}, $\theta$ is an unknown parameter that affects the dynamics and the Cauchy datum. We assume that $\theta$ varies in a compact set $\Theta$ embedded in some Euclidean space, and it parametrizes the systems of the ensemble.
We insist on the fact that, in our framework, the control $u:[0,T]\to\R^k$ is \emph{simultaneously} driving every element of the ensemble.

This standpoint is widely adopted in applications, typically when the controlled evolution depends on quantities obtained through (imprecise) measurements, or subject to statistical noise \cite{LK06, RuLi12}. In these circumstances, the user is interested in coming up with a single control that is suitable for every scenario (or many of them).
Also in linear-state systems, the controllability of ensembles is a source of inspiring questions, see \cite{LoZu16} and the more recent contributions \cite{DS21, Dan22, Schoe23}.
For ensembles of linear-control systems, in \cite{ABS16, AS1} Lie-algebraic arguments were employed to establish approximability results.
Moreover, ensembles are often considered in quantum control, see for example \cite{BCR10,BST15,AuBoSi18,ChiGau18} and the recent contribution \cite{RABS}, where the Authors address the control of a \emph{qubit} whose resonance frequence is not precisely known and ranges on an interval.
In addition, it is worth mentioning that ensembles of controls systems have been playing a central role in the mathematical modeling of Machine Learning.
Indeed, on the one hand, when considering \emph{neural ODEs}, the parameter $\theta$ is not directly influencing the dynamics, but it rather describes the elements of the \emph{data-set}, which are modeled as the Cauchy data for \eqref{eq:intro_ens} to be simultaneously steered to the corresponding targets (or \emph{labels}).
The controllability and the optimal control problems have  been investigated, respectively, in \cite{AS2} and \cite{Scag22}, while in \cite{CFS24} the mean-field perspective has been adopted, and in \cite{GZ22} the turnpike property of these systems was studied.
For a survey on neural ODEs, we recommend \cite{RuZu23}.
On the other hand, it has been shown that \emph{Reinforcement Learning} is intimately related to optimal control problems whose dynamics is partially unknown (see, e.g., \cite{MP18,PPF21}).
Finally, we mention that, in the very recent contributions \cite{GLPR23,AL24}, interactions between the evolving data-points have been introduced, in order to model the \emph{self-attention} mechanism of \emph{Transformers}.

In the present work, we consider ensembles of the form \eqref{eq:intro_ens} and we study minimax optimal control problems that can be formulated as follows:
\begin{equation}\label{eq:intro_minimax}
\begin{split}
&\min_{u\in\U} \J_\Theta(u),\\
\J_\Theta(u):= \max_{\theta\in\Theta} \, &a(x^\theta(T),\theta) + \int_0^T f(t,u(t))\,dt,
\end{split}
\end{equation}
where $x^\theta:[0,T]\to\R^d$ is the solution of \eqref{eq:intro_ens} corresponding to the parameter $\theta\in\Theta$ and to the control $u\in L^p([0,T],\R^k)$, while $a:\R^d\times \Theta\to[0,+\infty)$ is the end-point cost, and $f:[0,T]\times \R^k\to\R$ is a convex function designing the running cost.
In this paper, we aim at providing a thorough analysis of \eqref{eq:intro_minimax}. Namely, in Section~\ref{sec:existence} we first address the existence of minimizers using the direct method of Calculus of Variations (see Proposition~\ref{prop:existence}), and in Section~\ref{sec:Gamma} we employ $\Gamma$-convergence arguments to study the stability of the solutions of \eqref{eq:intro_minimax} when the set of parameters $\Theta$ is perturbed.
To this end, we use a strategy similar to the one followed in \cite{Scag23} for \emph{averaged} ensemble optimal control problems.
More precisely, given a sequence of compact sets $\Theta^1,\ldots,\Theta^\infty$ such that the Hausdorff distance $d_H(\Theta^N,\Theta^\infty)\to 0$ as $N\to\infty$, we establish a $\Gamma$-convergence result for the minimax problems associated to $\Theta^N$, $N\geq 1$ (see Theorem~\ref{thm:G_conv}).
In virtue of this, denoting with $\hat u_N$ a minimizer of $\J_{\Theta^N}$, it turns out that the sequence $(u_N)_{N}$ is pre-compact in the \emph{strong} topology of $L^p$, and every limiting point is a minimizer for $\J_{\Theta^\infty}$ (see Corollary~\ref{cor:strong_conv}).
On this point, we stress that, since we establish the $\Gamma$-convergence with respect to the weak topology of $L^p$, the standard theory would imply the pre-compactness of the minimizers in the $L^p$-\emph{weak} topology. However, using the celebrated result in \cite{Vis84}, we manage to strengthen the pre-compactness of the minimizers, whenever the running cost $f:[0,T]\times \R^k\to \R$ in \eqref{eq:intro_minimax} is strictly convex in the second variable. 
We report that a similar phenomenon has been observed in \cite{Scag22,Scag23} for the particular cost $f(u)=\beta |u|_2^2$, with $\beta>0$.
Finally, we mention that in \cite{PRG} the Authors proved a $\Gamma$-convergence result by working from the beginning with the \emph{strong topology}, in the framework of averaged ensemble optimal control problems. In this case, however, proving that functionals of the $\Gamma$-converging sequence are coercive in the underlying topology could be in general a hard task. \\
The $\Gamma$-convergence approach mentioned above encompasses as well the case where we approximate an infinite set of parameters $\Theta$ with a sequence of finite and increasing-in-size subsets.
This method of reducing to finite ensembles will be crucial for computing numerical solutions of \eqref{eq:intro_minimax}.
Moreover, in Section~\ref{sec:PMP}, this strategy is the cornerstone for deriving the necessary optimality conditions for the minimizers of $\J_\Theta$ in form of Pontryagin Maximum Principle (PMP), when the ensemble consists of infinitely many systems (see Theorem~\ref{thm:PMP_infinite}). 
We report that similar arguments ---though in a less general setting--- have been employed in \cite{Scag23}, where the final-time cost was \emph{averaged} over the elements of $\Theta$, and no maximization on the elements of the ensemble was involved.
In this regard, our result completes the minimax PMP for infinite ensembles established in \cite{V05}, where at every point the set of admissible velocities was supposed to be bounded, so that control-affine systems were not covered by the analysis.
Moreover, our techniques, which heavily rely on the fact that the dynamics is affine in the controls, differ from the ones adopted in \cite{V05}, based on Ekeland's Theorem. Finally, we mention that more recently the arguments of \cite{V05} were adapted to the \emph{averaged} cost case in \cite{BK19,BK24}, respectively without and with state-constraints.

\section{Preliminaries} \label{sec:prel}
Given a compact set $\Theta\subset \R^l$ and a time horizon $[0,T]$, we consider the following family of control systems in $\R^d$ parametrized by $\theta\in\Theta$: 
\begin{equation}\label{eq:ens_ctrl_sys}
\begin{cases}
\dot x^\theta (t) = F_0^\theta(\xth(t)) + F^\theta(\xth(t))u(t)&
\mbox{for a.e. }t\in [0,T],\\
\xth(0) = x_0(\theta),
\end{cases}
\end{equation}
where $F_0^\theta:\R^d\to\R^d$, $F^\theta=(F^\theta_1,\ldots,F^\theta_k):\R^d\to \R^{d\times k}$ define the vector fields of the ensemble of control-affine systems, and $x_0:\Theta\to \R^d$ prescribes the Cauchy data for the evolutions.
Using the same notations as in \cite{Scag23}, we introduce $F_0:\R^d\times\Theta\to\R^d$ and $F:\R^d\times\Theta\to\R^{d\times k}$ as:
\begin{equation}\label{eq:def_fields_maps}
F_0(x,\theta) := F_0^\theta(x), \qquad
F(x,\theta) = \big(F_1(x,\theta),\ldots,F_k(x,\theta) \big):= \big( F^\theta_1(x),\ldots,F_2^\theta(x) \big)
\end{equation}
for every $\theta\in\Theta$ and $x\in\R^d$. We require $F_0,F$ to satisfy the assumption below.
\begin{assumption}\label{ass:fields}
The applications $F_0$ and $F=(F_1,\ldots,F_k)$ defined as in \eqref{eq:def_fields_maps} are globally Lipschitz continuous, i.e.,  there exists $L>0$ such that 
\begin{equation*}
\sup_{i=0,\ldots,k} |F_i(x_1,\theta_1) - F_i(x_2,\theta_2)|_2 \leq 
L\big( |x_1-x_2|_2 + |x_1+x_2|_2 |\theta_1 - \theta_2|_2   \big)
\end{equation*}
for every $(x_1,\theta_1),(x_2,\theta_2)\in \R^d\times\Theta$.
\end{assumption}
We also make the following hypothesis on the function $x_0$.
\begin{assumption}\label{ass:cauchy}
The function $x_0:\Theta\to\R^d$ prescribing the initial conditions for the ensemble \eqref{eq:ens_ctrl_sys} is continuous.
\end{assumption}

For every $u\in L^1([0,T],\R^k)$ and $\theta\in\Theta$, the curve $x_u^\theta:[0,T]\to\R^d$ denotes the solution of the Cauchy problem \eqref{eq:ens_ctrl_sys} corresponding to the system identified by $\theta$ and to the control $u$. The existence of such a trajectory is guaranteed by the classical Carath\'eodory Theorem (see \cite[Theorem~5.3]{H80}).
We consider $\U:=L^p([0,1],\R^k)$ with $1<p<\infty$ as the space of  controls, and we equip it with the usual Banach space structure.
Given $u\in\U$, we describe the evolution of the ensemble of control systems \eqref{eq:ens_ctrl_sys} through the mapping $X_u:[0,T]\times\Theta\to\R^d$ defined as follows:
\begin{equation} \label{eq:def_evol_ens}
X_u(t,\theta) := x_u^\theta(t)
\end{equation}
for every $(t,\theta)\in[0,T]\times\Theta$.
With an argument based on the Gr\"onwall Lemma, we can prove the following result.

\begin{lemma} \label{lem:bound_X}
For every $u\in\U$,
let $X_u:[0,T]\times\Theta \to \R^d$ be the
application defined in \eqref{eq:def_evol_ens}
collecting the trajectories of the ensemble of
control systems \eqref{eq:ens_ctrl_sys}.
Then, for every $R>0$ there exists 
$\bar R>0$ such that, if $\|u\|_{L^p}\leq R$,
we have 
\begin{equation} \label{eq:bound_X}
|X_u(t,\theta)|_2\leq \bar R,
\end{equation}
for  every $(t,\theta)\in[0,T]\times \Theta$.
\end{lemma}

\begin{proof}
See \cite[Lemma~A.2]{Scag23}.
\end{proof}

The next convergence result is the backbone of several proofs in the present paper.

\begin{proposition} \label{prop:unif_conv_map_X}
Let us consider a sequence of  controls
$(u_n)_{n\in\NN} \subset \U$
such that $u_n\weak_{L^p}u_\infty$ as
$n\to\infty$.
For every $n\in \NN \cup \{ \infty\}$,
let $X_n:[0,T]\times \Theta\to \R^d$
be the application defined in
\eqref{eq:def_evol_ens} that collects the
trajectories of the ensemble of control
systems \eqref{eq:ens_ctrl_sys} 
corresponding to the  control $u_n$.
Then, we have that the family $(X_n)_{n\in\NN \cup\{\infty\}}$ is equi-bounded and equi-uniformly continuous, and
\begin{equation} \label{eq:unif_conv_map_X}
\lim_{n\to\infty}\, \sup_{(t,\theta)\in [0,T]\times\Theta} |X_n(t,\theta) - X_\infty(t,\theta)|_2
=0. 
\end{equation}
\end{proposition}

\begin{proof}
The equi-boundedness follows from Lemma~\ref{lem:bound_X}, while the equi-uniform continuity and the convergence descend from \cite[Lemma~A.5]{Scag23} and \cite[Proposition~2.2]{Scag23}, respectively.
\end{proof}

We report that in the proofs detailed in \cite{Scag23} the space of  controls $\U$ is set to be $L^2$. However, the arguments employed do not make use of the Hilbert space structure of $\U$, and they can be \textit{verbatim} generalized to the case $\U=L^p$ with $1<p<\infty$. 
Indeed, the crucial aspect lies in the reflexivity of $\U$ and of $W^{1,p}([0,T],\R^d)$ (i.e., the space where the trajectories solving \eqref{eq:ens_ctrl_sys} live), and the fact that $W^{1,p}([0,T],\R^d)$ is compactly embedded in $C^0([0,T],\R^d)$. 

\section{Minimax optimal control problem: existence result} \label{sec:existence}

In this paper, we consider the functional $\J:\U\to \Rext$ defined as
\begin{equation}\label{eq:def_fun_minimax}
\J(u) := \max_{\theta \in\Theta} a(X_u(T,\theta),\theta) + \int_0^T f(t,u(t))\, dt
\end{equation}
for every $u\in\U$. We require $a:\R^d\times \Theta \to \R_+$ and $f:[0,T]\times \R^k\to\R$ to satisfy the following assumptions.

\begin{assumption}\label{ass:terminal_cost}
The function $a:\R^d\times \Theta \to \R_+$ defining the terminal cost is continuous and non-negative.
\end{assumption}

\begin{assumption}\label{ass:running_cost}
The function $f:[0,T]\times \R^k \to \R$ defining the integral cost is such that:
\begin{enumerate}
\item[(i)] $f(\cdot,u)$ is measurable for every $u\in\R^k$.
\item[(ii)] $f(t,\cdot)$ is continuous and convex for a.e. $t\in [0,T]$.
\item[(iii)] There exist $c\in L^1([0,T],\R_+)$, $C>0$ and $p\in(1,\infty)$ such that 
\begin{equation} \label{eq:growth_f}
f(t,u) \geq C|u|^p_2 - c(t)
\end{equation}
for a.e. $t\in[0,T]$ and for every $u\in\R^k$.
\item[(iv)] There exist $u'\in\R^k$ and $c'\in L^1([0,T],\R_+)$ so that $|f(t,u')|\leq c'(t)$ for a.e. $t\in [0,T]$.
\end{enumerate}
\end{assumption}

\begin{remark}\label{rmk:exponent_space}
As mentioned in the previous section, we set $\U:=L^p([0,T],\R^k)$, where the exponent $p\in(1,\infty)$ matches with \eqref{eq:growth_f} in Hypothesis~\ref{ass:running_cost}. Indeed, this guarantees that the functional $\J$ is weakly coercive.
Moreover, we observe that the condition expressed in Hypothesis~\ref{ass:running_cost}-(iv) is a weaker relaxation of the more classical bound $|f(t,u)|\leq C''|u|_2^p + c''(t)$ for every $u\in\R^k$ and a.e. $t\in[0,T]$, with $C''>0$ and $c''\in L^1([0,T],\R_+)$. We recall that this last growth bound is necessary and sufficient for having that $t\mapsto f(t,u(t))\in L^1([0,T],\R)$ for every $u(\cdot)\in L^p([0,T],\R^k)$.
For more details, see \cite[Corollary~6.46]{FL07}.
\end{remark}

\begin{lemma}\label{lem:well_defined}
Under Hypotheses~\ref{ass:fields}-\ref{ass:running_cost} and setting $\U:=L^p([0,T],\R^k)$, we have that the functional $\J:\U\to\Rext$ introduced in \eqref{eq:def_fun_minimax} is well-defined at every $u\in\U$, and $\J\not\equiv+\infty$. 
\end{lemma}
\begin{proof}
We first observe that, owing to Lemma~\ref{lem:bound_X} and Hypothesis~\ref{ass:terminal_cost}, the first summand at the right-hand side of \eqref{eq:def_fun_minimax} is finite, for every $u\in\U$. 
Moreover, the growth condition in Hypothesis~\ref{ass:running_cost}.iii ensures that the integral term is never equal to $-\infty$ for $u\in\U$ (even if it can be equal to $+\infty$).
Finally, in virtue of Hypothesis~\ref{ass:running_cost}.iv, we consider the control $\tilde u(t)= u'$ for every $t\in[0,T]$, and we obtain that $\J(\tilde u)<+\infty$.
\end{proof}

We first establish an ancillary result concerning the terminal-time cost at the right-hand side of \eqref{eq:def_fun_minimax}.

\begin{lemma}\label{lem:conv_maxima}
Under Hypotheses~\ref{ass:fields}-\ref{ass:terminal_cost} and setting $\U:=L^p([0,T],\R^k)$ with $1<p<\infty$, let us consider a sequence of  controls $(u_n)_{n\in\NN} \subset \U$ such that $u_n\weak_{L^p}u_\infty$ as $n\to\infty$.
For every $n\in \NN \cup \{ \infty\}$, let $X_n:[0,T]\times \Theta\to \R^d$ be the application defined in \eqref{eq:def_evol_ens} that collects the trajectories of the ensemble of control systems \eqref{eq:ens_ctrl_sys} corresponding to the  control $u_n$.
Then,
\begin{equation} \label{eq:lim_max_conv}
\lim_{n\to\infty}\, \max_{\theta \in\Theta} a(X_n(T,\theta),\theta) = \max_{\theta \in\Theta} a(X_\infty(T,\theta),\theta).
\end{equation}
\end{lemma}
\begin{proof}
Let us take $\theta^*\in \arg \max_{\theta \in\Theta} a(X_\infty(T,\theta),\theta)$, which exists in view of Hypothesis~\ref{ass:terminal_cost}, the continuity of $X_\infty$, and the compactness of $\Theta$.
Owing to Proposition~\ref{prop:unif_conv_map_X}, we deduce that
\begin{equation*}
\lim_{n\to\infty}  a(X_n(T,\theta^*),\theta^*) =
a(X_\infty(T,\theta^*),\theta^*).
\end{equation*}
Since $\max_{\theta \in \Theta}a(X_n(T,\theta),\theta)\geq a(X_n(T,\theta^*),\theta^*)$ for every $n\geq 1$, we finally obtain that
\begin{equation}\label{eq:liminf_max_conv}
\liminf_{n\to\infty}\, \max_{\theta \in \Theta}a(X_n(T,\theta),\theta) \geq \lim_{n\to\infty}  a(X_n(T,\theta^*),\theta^*) 
%=a(X_\infty(T,\theta^*),\theta^*)
 = \max_{\theta \in \Theta}a(X_\infty(T,\theta),\theta).
\end{equation}
Now, for every $n\geq1$, we consider $\theta^*_n\in \arg\max_{\theta\in \Theta}a(X_n(T,\theta),\theta)$, and we extract a subsequence such that
\begin{equation*}
\lim_{j\to\infty} a(X_{n_j}(T,\theta^*_{n_j}),\theta^*_{n_j})
= 
\limsup_{n\to\infty}
a(X_{n}(T,\theta^*_{n}),\theta^*_{n})
=
\limsup_{n\to\infty}\, \max_{\theta\in \Theta}a(X_{n}(T,\theta),\theta)
\end{equation*}
and $\theta^*_{n_j}\to\theta^*_\infty$ as $j\to\infty$.
From the weak convergence $u_n\weak_{L^p}u_\infty$, we deduce that there exists $R>0$ such that $\|u_n\|_{L^p}\leq R$ for every $n\in \NN \cup \{ \infty\}$. Hence, using Lemma~\ref{lem:bound_X}, we obtain that $X_n(T,\theta)\in B_{\bar R}(0)$ for every $\theta\in \Theta$  and for every $n\in \NN \cup \{ \infty\}$.
%we notice that the restriction $a|_{B_{\bar R}(0)\times \Theta}$ is uniformly continuous.
Therefore, recalling that $a$ is continuous, we conclude that
\begin{equation*}
\lim_{j\to\infty} a(X_{n_j}(T,\theta^*_{n_j}),\theta^*_{n_j})
= a(X_{\infty}(T,\theta^*_{\infty}),\theta^*_{\infty}).
\end{equation*}
Moreover, from the last two identities it follows that
\begin{equation}\label{eq:limsup_max_conv}
\max_{\theta\in \Theta}a(X_{\infty}(T,\theta),\theta)
\geq a(X_{\infty}(T,\theta^*_{\infty}) ,\theta^*_{\infty}) =
\limsup_{n\to\infty}\, \max_{\theta\in \Theta}a(X_{n}(T,\theta),\theta).
\end{equation}
Therefore, combining \eqref{eq:liminf_max_conv} and \eqref{eq:limsup_max_conv}, we deduce \eqref{eq:lim_max_conv} and we conclude the proof.
\end{proof}

\begin{remark}\label{rmk:cluster_maximizers}
Using the same notations as in the proof of Lemma~\ref{lem:conv_maxima}, we observe that \eqref{eq:liminf_max_conv} and \eqref{eq:limsup_max_conv} imply $\theta^*_{\infty} \in \arg \max_{\theta\in \Theta}a(X_{\infty}(T,\theta),\theta)$.
Therefore, as a byproduct of the argument in the previous proof, we obtain that, if $\theta^*_n$ is in $ \arg\max_{\theta\in \Theta}a(X_n(T,\theta),\theta)$, then any accumulation point of the sequence $(\theta^*_n)_{n\geq1}$ belongs to  $\arg \max_{\theta\in \Theta}a(X_{\infty}(T,\theta),\theta)$.
\end{remark}

In the next result, we show that the minimization problem involving the functional $\J$ admits solutions.

\begin{proposition}\label{prop:existence}
Under Hypotheses~\ref{ass:fields}-\ref{ass:running_cost} and setting $\U:=L^p([0,T],\R^k)$, the functional $\J:\U\to \R$ defined in \eqref{eq:def_fun_minimax} assumes minimum, i.e., there exists $\hat u\in\U$ such that
\begin{equation*}
\J(\hat u) = \inf_{u\in\U} \J(u).
\end{equation*} 
\end{proposition}
\begin{proof}
The argument is based on the direct method of Calculus of Variations (see \cite[Theorem~1.15]{D93}), i.e., we shall prove that the functional $\J$ is sequentially lower semi-continuous and coercive with respect to the weak topology of $L^p$. \\
In regards of the lower semi-continuity, Lemma~\ref{lem:conv_maxima} guarantees that the end-point cost at the right-hand side of \eqref{eq:def_fun_minimax} is sequentially continuous along weakly convergent sequences of controls. 
Therefore, it suffices to show that the integral term is sequentially weakly lower semi-continuous. However, this fact follows from Hypothesis~\ref{ass:running_cost} and from a classical result for weak lower semi-continuity of integral functionals in $L^p$ (see \cite[Theorem~6.54]{FL07}). \\
To establish the $L^p$-weak coercivity, we observe that
for every $K\in\R$ we have that
\begin{equation*}
K \geq \J(u) \implies  \| u \|_{L^p}^p \leq  \frac{K+\|c\|_{L^1}}{C},
\end{equation*}
where we used Hypothesis~\ref{ass:running_cost}-(iii) and the non-negativity of the function $a$ designing the end-point cost. This concludes the proof.
\end{proof}

\section{$\Gamma$-convergence for minimax problems} \label{sec:Gamma}

In this section, we consider a family of functionals
$\J^N:\U\to \R$ defined as
\begin{equation}\label{eq:def_fun_minimax_N}
\J^N(u) := \max_{\theta \in\Theta^N} a(X_u(T,\theta),\theta)  + \int_0^T f(t,u(t))\, dt
\end{equation}
for every $u\in\U$, where $\Theta^N\subseteq \Theta$ is compact for every $N\in\Next$.
In particular, we are interested in studying the \textit{stability} of the minimization problems related to \eqref{eq:def_fun_minimax_N} in the case that $d_H(\Theta^N,\Theta^\infty)\to 0$ as $N\to\infty$, where $d_H(\cdot,\cdot)$ denotes the Hausdorff distance between sets. Recalling that we have $\Theta\subset\R^l$ compact, in our framework the Hausdorff distance has the following expression
\begin{equation*}
d_H(A,B) := \max\left\{ \sup_{a\in A}\inf_{b\in B}|a-b|_2,\,  \sup_{b\in B}\inf_{a\in A}|b-a|_2 \right\}
\end{equation*} 
for every $A,B\subset \Theta$. For further details on this topic, we recommend \cite{Haus99}.
\begin{remark}\label{rmk:finite_approx}
A particular example relevant for applications is when $\Theta^\infty=\Theta$ and $\Theta^N\subset \Theta$ is a finite $\e_N$-net for $\Theta$, i.e., for every $\theta \in \Theta$ there exists $\theta'\in \Theta^N$ such that $|\theta-\theta'|_2\leq \e_N$.
If $(\e_N)_{N\in \NN}$ is a sequence such that $\e_N \to 0$ as $N\to\infty$, then $d_H(\Theta^N,\Theta^\infty)\to 0$.
In this case, we are interested in minimizing the functional $\J^\infty$, whose single evaluation, however, requires a maximization problem over an infinite set $\Theta$, and it involves the resolution of an infinite number of Cauchy problems.
On the other hand, for $N<\infty$ the maximum that appears in the definition of $\J^N$ is taken over a finite set.
The aim of this section is to understand in what sense the problem of minimizing $\J^N$ provides an approximation to the problem of minimizing the more complicated functional $\J^\infty$.

\end{remark}

Before proceeding to the main result of this section, we prove an extension of Lemma~\ref{lem:conv_maxima}.

\begin{lemma}\label{lem:conv_maxima_gen}
Under Hypotheses~\ref{ass:fields}-\ref{ass:terminal_cost} and setting $\U:=L^p([0,T],\R^k)$ with $1<p<\infty$, let us consider a sequence of  controls $(u_N)_{N\in\NN} \subset \U$ such that $u_N\weak_{L^p}u_\infty$ as $N\to\infty$.
For every $N\in \NN \cup \{ \infty\}$, let $X_N:[0,T]\times \Theta\to \R^d$ be the application defined in \eqref{eq:def_evol_ens} that collects the trajectories of the ensemble of control systems \eqref{eq:ens_ctrl_sys} corresponding to the  control $u_N$.
Moreover, let us further assume that for every $N\in\Next$ we have a compact set $\Theta^N\subseteq\Theta$ such that 
%the sequence of sets is converging in the Hausdorff distance, i.e., 
$d_H(\Theta^N,\Theta^\infty)\to 0$ as $N\to\infty$.
Then,
\begin{equation} \label{eq:lim_max_conv_gen}
\lim_{N\to\infty}\, \max_{\theta \in\Theta^N} a(X_N(T,\theta),\theta) = \max_{\theta \in\Theta^\infty} a(X_\infty(T,\theta),\theta).
\end{equation}
\end{lemma}
\begin{proof}
The argument follows the lines of the proof of Lemma~\ref{lem:conv_maxima}. 
Using Proposition~\ref{prop:unif_conv_map_X}, it descends that the mappings $\theta\mapsto a(X_N(T,\theta),\theta)$ are converging uniformly for $\theta\in\Theta$ to $\theta\mapsto a(X_\infty(T,\theta),\theta)$.
Moreover, given $\theta^*\in \arg \max_{\theta \in\Theta^\infty} a(X_\infty(T,\theta),\theta)$, and, for every $N\geq 1$, we choose $\theta'_N\in \Theta^N$ such that $|\theta^* - \theta'_N|_2\leq\e_N:=d_H(\Theta^N,\Theta^\infty)$. Then, we have that
\begin{equation*}%\label{eq:liminf_max_conv}
\liminf_{N\to\infty}\, \max_{\theta \in \Theta^N}a(X_N(T,\theta),\theta) \geq \lim_{N\to\infty}  a(X_N(T,\theta'_N),\theta'_N) 
%=a(X_\infty(T,\theta^*),\theta^*)
 = \max_{\theta \in \Theta^\infty}a(X_\infty(T,\theta),\theta).
\end{equation*}
With a \textit{verbatim} repetition of the passages done in Lemma~\ref{lem:conv_maxima}, we obtain that
\begin{equation*}
\max_{\theta\in \Theta^\infty}a(X_{\infty}(T,\theta),\theta)
\geq 
\limsup_{N\to\infty}\, \max_{\theta\in \Theta^N}a(X_{N}(T,\theta),\theta),
\end{equation*}
and we deduce the thesis.
\end{proof}

\begin{remark}\label{rmk:cluster_maxim_2}
Similarly as observed in Remark~\ref{rmk:cluster_maximizers}, if $u_N\weak_{L^p} u_\infty$ as $N\to\infty$ and $\theta^*_N\in \arg\max_{\theta\in \Theta^N}a(X_N(T,\theta),\theta)$, then any accumulation point of the sequence $(\theta^*_N)_{N\geq1}$ belongs to  $\arg \max_{\theta\in \Theta}a(X_{\infty}(T,\theta),\theta)$. 
\end{remark}

We shall study the asymptotics of the minimization of $\J^N$ through the lens of $\Gamma$-convergence. We briefly recall below this notion. For a thorough introduction to this topic, we refer the reader to the textbook \cite{D93}. 

\begin{defn} \label{defn:G_conv}
Let $(\X,d)$ be a metric space, and for every $N\geq 1$ let $\F^N:\X\to\R\cup \{+\infty \}$ be a functional defined over $X$. 
The sequence $(\F^N)_{N\geq 1}$ is said to $\Gamma$-converge to a functional $\F:\X\to\R\cup \{+\infty \}$ if the following conditions holds:
\begin{itemize}
\item \emph{liminf condition}: for every sequence
$(u_N)_{N\geq 1}\subset \X$ such that $u_N\to_\X u$
as $N\to\infty$ the following inequality holds
\begin{equation}\label{eq:liminf_cond}
\F(u) \leq \liminf_{N\to\infty} \F^N(u_N);
\end{equation}
\item \emph{limsup condition}: 
for every $u\in \X$ there exists a sequence 
$(u_N)_{N\geq 1}\subset \X$ such that 
$u_N\to_\X u$ as $N\to\infty$ and such that
the following inequality holds:
\begin{equation}\label{eq:limsup_cond}
\F(u) \geq \limsup_{N\to\infty} \F^N(u_N).
\end{equation}
\end{itemize}
If the conditions listed above are satisfied, then
we write $\F^N \to_\Gamma \F$ as $N\to\infty$.
\end{defn} 

In view of $\Gamma$-convergence, it is possible to relate the minimizers of the functionals $(\F^N)_{N\geq 1}$ to the minimizers of the limiting functional $\F$. Namely, assuming that the functionals of the sequence $(\F^N)_{N\geq 1}$ are equi-coercive, if $\hat u_N \in \mathrm{arg\, min}\, \F^N$ for every $N\geq 1$, then the sequence $(\hat u_N)_{N\geq1}$ is pre-compact in $\X$, and any of its limiting points is a minimizer for $\F$ (see \cite[Corollary~7.20]{D93}).
This means that the problem of minimizing $\F$ can be replaced by the minimization of $\F^N$, when $N$ is sufficiently large.\\
In our framework, it is convenient to endow the space of controls $\U:=L^p([0,1],\R^k)$ ($1<p<\infty$) with the weak topology. However, on the one hand, Definition~\ref{defn:G_conv} requires the domain $\X$ where the limiting and the approximating functionals are defined to be a metric space. On the other hand, the weak topology of $L^p$ is metrizable only when restricted to bounded sets (see \cite[Theorem~3.29 and Remark~3.3]{B11}).
In the next lemma, we see how we can safely choose a restriction on a bounded set, without losing any minimizer.

\begin{lemma}\label{lem:restr}
Under Hypotheses~\ref{ass:fields}-\ref{ass:running_cost} and setting $\U:=L^p([0,T],\R^k)$, for every $N\in\Next$, let $\J^N:\U\to\R_+$ be the functional defined in
\eqref{eq:def_fun_minimax_N}, where the sequence of compact sets $\Theta^1,\ldots,\Theta^\infty \subseteq\Theta$ satisfies 
%the sequence of sets is converging in the Hausdorff distance, i.e., 
$d_H(\Theta^N,\Theta^\infty)\to 0$ as $N\to\infty$. Then, there exists 
$\rho>0$ such that, if $ u^* \in \arg \min_\U \, \J^N$ for some $N\in \Next$, then
\begin{equation} \label{eq:rad_restr}
\| u^* \|_{L^p} \leq \rho.
\end{equation} 
\end{lemma}

\begin{proof}
Let us consider the control $\bar u(t) = u'$ for every $t\in[0,T]$, where $u'\in \R^k$ is the value prescribed by Hypothesis~\ref{ass:running_cost}-(iv). If $ u^* \in \U$ is a minimizer for  the functional $\J^N$, then we have $\J^N(u^*) \leq \J^N(\bar u)$. Then, using Lemma~\ref{lem:bound_X} and the inclusion $\Theta^N\subseteq\Theta$, we deduce:
\begin{equation*}
\J^N(\bar u) \leq \kappa:= \max_{\theta\in\Theta} 
a(X_{\bar u}(T,\theta),\theta) + \|c'\|_{L^1}.
\end{equation*}
Moreover, recalling that the function $a:\R^d\times \Theta\to\R_+$ that designs the terminal cost in \eqref{eq:def_fun_minimax} is non-negative, from Hypothesis~\ref{ass:running_cost}-(iii) we obtain
\[
C\|u^*\|_{L^p}^p - \| c \|_{L^1} \leq \J^N( u^*)
\leq \kappa.
\] 
Thus, to prove \eqref{eq:rad_restr} it is sufficient
to set $\rho:= \left(\frac{\kappa + \|c\|_{L^1}}{C}\right)^{\frac1p}$.
\end{proof}

From the previous result, we deduce the following inclusion:
\begin{equation*}
\bigcup_{N=1}^\infty \left(\arg \min_\U \, \J^N\right) \subset \X,
\end{equation*}
where we set 
\begin{equation} \label{eq:def_restr_sp}
\X:=\{ u\in \U: \|u\|_{L^p} \leq\rho\},
\end{equation}
and where $\rho>0$ is provided by Lemma~\ref{lem:restr}.
Since $\X$ is a closed ball of $L^p$, when it is equipped with the $L^p$-weak topology, $\X$ is a compact metric space. Hence, for every $N\in\Next$ we can restrict the functionals $\J^N:\U\to\R$ to $\X$ to construct an approximation in the sense of $\Gamma$-convergence. With a slight abuse of notations, we shall continue to denote by $\J^N$ the functional restricted to $\X$.

\begin{theorem}\label{thm:G_conv}
Under Hypotheses~\ref{ass:fields}-\ref{ass:running_cost}, let $\X\subset \U$ be the set defined in \eqref{eq:def_restr_sp}, equipped with the weak topology of $L^p$. 
For every $N\in\Next$, let $\J^N:\X\to\R$ be the restriction to $\X$ of the applications defined in \eqref{eq:def_fun_minimax_N}, where $\Theta^N\subseteq\Theta$.  
Then, if $d_H(\Theta^N,\Theta^\infty)\to0$ as $N\to\infty$, we have $\J^N\to_\Gamma \J^\infty$ and 
\begin{equation} \label{eq:conv_inf}
\inf_\X \J^\infty = \lim_{N\to\infty} \inf_\X\J^N.
\end{equation} 
Moreover, if $ u_N^* \in \X$ is a minimizer of $\J^N$ for every $N\in\NN$, then the sequence
$( u_N^*)_{N\in\NN}$ is pre-compact with
respect to the weak topology of $L^p$, and 
any limiting point of this sequence is a minimizer
of $\J^\infty$.
\end{theorem}

\begin{proof}
We start by establishing the $\liminf$ inequality \eqref{eq:liminf_cond}. Let us consider a sequence of  controls $(u_N)_{N\geq 1}\subset \X$ such that $u_N\weak_{L^p} u_\infty$ as $N\to \infty$.
Then, recalling the characterization of weakly lower semi-continuous functionals in $L^p$ (see \cite[Theorem~6.54]{FL07}), Hypothesis~\ref{ass:running_cost} ensures that
\begin{equation*}
\int_0^T f(t,u_\infty(t))\, dt \leq \liminf_{N\to\infty}
\int_0^T f(t,u_N(t))\, dt,
\end{equation*}
and, in virtue of Lemma~\ref{lem:conv_maxima_gen}, we deduce that
\begin{equation*}
\J^\infty(u_\infty) 
\leq  \liminf_{N\to \infty}\J^N( u_N ).
\end{equation*}
Regarding the $\limsup$ inequality, given an  control $u\in\X$, we consider the constant sequence by setting $u_N:= u$ for every $N\geq 1$. Using again Lemma~\ref{lem:conv_maxima_gen}, we immediately deduce that
\begin{equation*}
\J^\infty(u) = \lim_{N\to \infty}\J^N( u ),
\end{equation*} 
establishing the $\Gamma$-convergence.
Finally, the convergence of the minimizers and of the infima is a well-known fact (see \cite[Corollary~7.20]{D93}).
\end{proof}

The next result is a direct consequence of the $\Gamma$-convergence result established in Theorem~\ref{thm:G_conv}. 

\begin{corollary} \label{cor:conv_integral_term}
Under the same hypotheses and notations as in Theorem~\ref{thm:G_conv}, for every $N\in\NN$ let $ u_N^* \in \X$ be a minimizer of $\J^N$, and let us assume that there exists $u_\infty^*\in\X$ such that $u_N^*\weak_{L^p}u_\infty^*$. 
Then, we have 
\begin{equation} \label{eq:conv_integral_term}
\int_0^Tf(t,u_\infty^*(t))\, dt = \lim_{N\to\infty} \int_0^Tf(t,u_N^*(t))\, dt.
\end{equation} 
\end{corollary}

\begin{proof}
From Theorem~\ref{thm:G_conv}, it follows immediately that $u_\infty^*$ is a minimizer of $\J^\infty$.
Hence, from \eqref{eq:conv_inf} we deduce that
\begin{equation*}
\J^\infty(u^*_\infty) = \lim_{N\to\infty} \J^N(u^*_N),
\end{equation*}
and by combining this identity with the definition of $\J^N$ in \eqref{eq:def_fun_minimax_N} and with Lemma~\ref{lem:conv_maxima_gen}, we obtain the thesis.
\end{proof}

In the special case where the function defining the running cost is of the form $f(t,u)= \beta |u|_2^p$ with $\beta>0$, given a $L^p$-weakly convergent sequence of minimizers $u^*_N\weak_{L^p}u^*_\infty$, then the identity \eqref{eq:conv_integral_term} in Corollary~\ref{cor:conv_integral_term} yields
\begin{equation*}
\|u^*_\infty\|_{L^p} = \lim_{N\to\infty}\|u^*_N\|_{L^p}.
\end{equation*}
Recalling that the spaces $L^p$ are uniformly convex for $1<p<\infty$, the last limit implies that the sequence of minimizers $(u^*_N)_{N\in\NN}$ is actually \textit{strongly} convergent to $u^*_\infty \in \arg\min_\U \J^\infty$ (see \cite[Proposition~3.32]{B11}).
Hence, despite that usually the approximation of the minimizers holds for the same topology that underlies the $\Gamma$-convergence result, here the approximation is provided with respect to the \emph{strong} topology of $L^p$, and not just in the weak sense.
A natural question is when, in general, we may have pre-compactness of the minimizers of $(\J^N)_{N\in\NN}$ in the $L^p$-strong topology. 
Before answering, we formulate a further assumption.

\begin{assumption}\label{ass:strict_conv}
The function $f:[0,T]\times\R^k \to\R$ defining the running cost in \eqref{eq:def_fun_minimax} and in \eqref{eq:def_fun_minimax_N} is strictly convex in the second argument, for a.e. $t\in[0,T]$.
\end{assumption}

The convergence in the strong topology holds whenever the running cost is strictly convex in the control variable, and this follows from a result established in the celebrated work \cite{Vis84}.
This condition generalizes the situation described in \cite{Scag22,Scag23} for optimal control of ensembles with \textit{averaged terminal cost}.

\begin{corollary}\label{cor:strong_conv}
Under the same hypotheses and notations as in Theorem~\ref{thm:G_conv}, for every $N\in\NN$ let $ u_N^* \in \X$ be a minimizer of $\J^N$. In addition, let us suppose that Hypothesis~\ref{ass:strict_conv} holds.
Then, the sequence $(u^*_N)_{N\in\NN}$ is pre-compact in the $L^p$-strong topology, and every limiting point is a minimizer of $\J^\infty$.
\end{corollary}
\begin{proof}
Up to a subsequence, we may assume that $u^*_N\weak_{L^p}u^*_\infty$ as $N\to\infty$, and the limit is a minimizer of $\J^\infty$ owing to Theorem~\ref{thm:G_conv}.
In view of Hypothesis~\ref{ass:strict_conv}, we have that, for a.e. $t\in[0,T]$,
the point $\left(u^*_\infty(t),f(t,u^*_\infty(t))\right)\in\R^k\times\R$ is an extremal point for the epigraph of the function $w\mapsto f(t,w)$. Therefore, Corollary~\ref{cor:conv_integral_term} guarantees that the hypotheses of \cite[Theorem~3]{Vis84} are satisfied, yielding $\|u^*_N - u^*_\infty\|_{L^p}\to 0$ as $N\to\infty$.
\end{proof}

\begin{remark} \label{rmk:comment_conv}
    We report that the convergence result concerning the minimizers $(u^*_N)_{N\in\NN}$ of the functionals $\J^N$ has a \emph{qualitative nature}, in the sense that we are not currently able to provide a rate for the decay of the quantity $\inf_{u^*_\infty\in\mathcal{M}}\|u^*_N-u^*_\infty\|_{L^p}$, where $\mathcal{M}:=\arg\min_\U \J^\infty$.
    We expect that some additional hypotheses are required, like for example that $(u^*_N)_{N\in\NN}$ is contained in a region where the functionals $(\J^N)_{N\in\NN}$ are uniformly strongly convex. We leave this open question for further investigation.\\
    Another intriguing point concerns the dimension of set of parameters $\Theta$. Indeed, if $\dim (\Theta)=l$, we have that the size of an $\varepsilon$-net is $O(\varepsilon^{-l})$, and this exponential scaling is known in literature as the \emph{curse of dimensionality}. A possible way for mitigating this phenomenon could be the introduction of randomness in the construction of $\varepsilon$-nets. 
    We plan to explore this approach in future studies.
\end{remark}

The last result shall play a crucial role in the next section, where we establish the Pontryagin Maximum Principle for minimax problems involving infinite ensembles.

\section{Pontryagin Maximum Principle via $\Gamma$-convergence} \label{sec:PMP}

In this section, we derive necessary optimality conditions in the form of Pontryagin's Maximum Principle (PMP) for local minimizers of the functional $\J:\U\to\R$ defined as in \eqref{eq:def_fun_minimax}. Here, we set $\U:=L^p([0,T],\R^k)$, and the exponent $p\in(1,\infty)$ is chosen according to the growth condition of the running-cost integrand (see Hypothesis~\ref{ass:running_cost}-(iv)).
The main result of this part is the establishment of the PMP in the case the set $\Theta$ consists of infinitely-many points.
Following a similar strategy as in \cite{Scag23} (where only problems with averaged-cost were considered), we deduce the PMP with a limiting argument based on the $\Gamma$-convergence and on the sequential approximation of $\Theta$ with growing-in-size finite sets $\Theta^N\subset \Theta$. 
This argument allows us to extend to control-affine systems the PMP for minimax optimal control problems originally established in \cite{V05}. 
Indeed, on the one hand, we report that the minimax PMP established in \cite{V05} does not encompass control systems with unbounded set of admissible velocities (as it is the case for control-affine systems). 
On the other hand, the proof presented here relies on the affine dependence of the systems in the control variables.

%Indeed, on the one hand, even though nonaffine-control systems are embraced in their analysis, in \cite[Assumption~S2]{V05} the set of allowed velocities is assumed to be bounded at every point. On the other hand, we report that the proof presented here strongly relies on the affine dependence in the controls.

Before proceeding, we need to formulate a further regularity assumption on the dynamics and on the terminal-time cost.

\begin{assumption}\label{ass:C1_reg}
The state-derivatives $(x,\theta)\mapsto \frac{\partial}{\partial x} F_0(x,\theta)$ and $(x,\theta)\mapsto \frac{\partial}{\partial x} F(x,\theta)$ of the maps defined in \eqref{eq:def_fields_maps} are continuous, as well as the gradient $(x,\theta)\mapsto \nabla_x a(x,\theta)$ of the function that designs the end-point cost in \eqref{eq:def_fun_minimax}.
\end{assumption}

For every $u\in \U$ and every $\theta\in \Theta$, we define the function $\lambda_u^\theta:[0,T]\to (\R^d)^*$ as the solution of the following differential equation
\begin{equation} \label{eq:def_lambda_theta_u}
\begin{cases}
\dot \lambda_u^\theta(t) = - \lambda_u^\theta(t)
\left( 
\frac{\partial F_0(x_u^\theta(t), \theta)}{\partial x}
+ \sum_{i=1}^ku_i(t) 
\frac{\partial F_i(x_u^\theta(t), \theta)}{\partial x} 
\right),\\
\lambda_u^\theta(T) = -\nabla_x a(x_u^\theta(T)),\theta),
\end{cases}
\end{equation}
where the curve $x_u^\theta:[0,1]\to\R^d$ is the solution of the Cauchy problem \eqref{eq:ens_ctrl_sys} corresponding to the system identified by $\theta$ and to the  control $u$. 
We insist on the fact that in this paper (except where explicitly specified) $\lambda_u^\theta$ is always understood as a row-vector, as well as any other element of $(\R^d)^*$. The existence and the uniqueness of the solution of \eqref{eq:def_lambda_theta_u} follow as a standard application of the Carath\'eodory Theorem (see, e.g., \cite[Theorem~5.3]{H80}).
Similarly as done in the previous subsection, for every $u\in\U$ we introduce the function $\Lambda_u:[0,T]\times \Theta\to (\R^d)^*$ defined as 
\begin{equation} \label{eq:def_Lambda}
\Lambda_u(t,\theta):= \lambda_u^\theta(t).
\end{equation}
In the case of a sequence of weakly convergent controls $(u_n)_{n\in\NN}$, for the corresponding sequence $(\Lambda_{u_n})_{n\in\NN}$ we can establish a result analogue to Proposition~\ref{prop:unif_conv_map_X}.
Exactly as it was for Proposition~\ref{prop:unif_conv_map_X}, the proof follows from an immediate adaptation to $\U=L^p$ ($1<p<\infty$) of the case $\U=L^2$ detailed in \cite[Section~2.3]{Scag23}.

\begin{proposition} \label{prop:unif_conv_map_Lambda}
Under Hypotheses~\ref{ass:fields}-\ref{ass:running_cost} and~\ref{ass:C1_reg}, let $(u_n)_{n\in\NN} \subset \U$ be a sequence of controls such that $u_n\weak_{L^p}u_\infty$ as $n\to\infty$.
For every $n\in \Next$, let $\Lambda_n:[0,T]\times \Theta\to (\R^d)^*$ be the application defined in \eqref{eq:def_Lambda} corresponding to the  control $u_n$.
Then, we have that
\begin{equation} \label{eq:unif_conv_map_Lambda}
\lim_{n\to\infty}\, \sup_{(t,\theta)\in [0,T]\times\Theta} |\Lambda_n(t,\theta) - \Lambda_\infty(t,\theta)|_2
=0. 
\end{equation}
\end{proposition}
\begin{proof}
See \cite[Proposition~2.5]{Scag23}.
\end{proof}

We recall below the necessary optimality conditions in the case of a finite ensemble of control systems.
Given a Polish space $A$, we denote by $\mathcal{P}(A)$ the space of Borel probability measures supported in $A$.
We recall that the compact set $\Theta\subset\R^l$ is equipped with the usual Euclidean distance.
Moreover, in case of a finite subset $\Theta^N\subset \Theta$, we put on $\Theta^N$ the topology induced by $\Theta$, i.e. the discrete topology. In this way, denoting with $\iota:\Theta^N\to\Theta$ the inclusion, we can embed every $\mu_N\in\mathcal{P}(\Theta^N)$ into $\mathcal{P}(\Theta)$ using the push-forward $\iota_\#\mu_N$. With a slight abuse of notations, we shall omit $\iota_\#$ also when we understand $\mu_N\in\mathcal{P}(\Theta)$.

\begin{theorem}\label{thm:PMP_finite}
Under Hypotheses~\ref{ass:fields}-\ref{ass:running_cost} and~\ref{ass:C1_reg}, let $u^*_N\in \U$ be a local minimizer for the functional $\J^N:\U\to\R$ defined in \eqref{eq:def_fun_minimax_N}, where $\Theta^N\subset \Theta$ consists of \emph{finitely many elements}.
Finally, let $X_N:[0,T]\times \Theta\to\R^d$ and $\Lambda_N:[0,T]\times \Theta\to(\R^d)^*$ be the mappings defined, respectively, in \eqref{eq:def_evol_ens} and in \eqref{eq:def_Lambda}. 
Then, there exists a probability measure $\mu_N\in\mathcal{P}(\Theta^N)$ such that
\begin{equation}\label{eq:max_cond_N}
u^*_N(t) \in \arg \max_{\!\!\!\!\!\!
\!\!\!\!\!\! v\in \R^k} \left\{
\sum_{\theta\in \Theta^N} \mu_N(\{\theta\}) 
\Lambda_N(t,\theta)\cdot F(X_N(t,\theta),\theta)\cdot v - f(t,v)
\right\},
\end{equation}
and
\begin{equation}\label{eq:supp_N}
\theta \in \mathrm{supp}(\mu_N) \implies
a(X_N(T,\theta),\theta)=
\max_{\vartheta\in\Theta^N} a(X_N(T,\vartheta),\vartheta).
\end{equation}
\end{theorem}
\begin{proof}
The result follows as an application of \cite[Proposition~2.1]{V05}.
Namely, we first define for every $\theta\in \Theta^N$ a new trajectory in $\R$ satisfying: 
\begin{equation*}
    \begin{cases}
    \dot y^\theta_u(t) = f(t,u(t)) & \mbox{a.e. in }[0,T],\\
    y^\theta_u(0) = 0,
    \end{cases}
\end{equation*}
and then we rephrase $\J^N(u)= \max_{\theta\in \Theta^N} g \big( (x_u^\theta(T),y_u^\theta(T)), \theta \big)$, where $\big((x,y),\theta \big)\mapsto g\big((x,y),\theta \big):=a(x,\theta) + y$, in order to deal only with terminal-state costs. 
Moreover, adapting the notations employed in \cite{V05} with ours, we have the following expression for the Hamiltonian:
\begin{equation}\label{eq:hamiltonian_proof}
    H\big( t, (x,y),(\lambda,\bar \lambda), u, \theta \big) = \lambda \cdot  (F_0(x,\theta) + F(x,\theta)u) + \bar\lambda f(t,u),
\end{equation}
where $\bar\lambda\in \R$ is the extra adjoint variable associated to the newly-introduced state $y$.
Moreover, following \cite[Proposition~2.1]{V05}, we set 
\begin{equation}\label{eq:function_G}
    G\big((x,y),\theta \big) := 
    \begin{cases}
        (\nabla_x a(x,\theta), 1) & \mbox{if } a(x,\theta) = \max_{\vartheta\in \Theta^N} a(x,\vartheta),\\
        \emptyset & \mbox{otherwise},
    \end{cases}
\end{equation}
where we used that $\theta \in \arg \max_{\vartheta\in \Theta^N} a(x,\vartheta) \iff \theta\in\arg \max_{\vartheta\in \Theta^N} g\big((x,y),\vartheta \big)$.
Then, since there are no terminal-state constraints, we have that $N\big((x,y),\theta \big)= \emptyset$ for every $(x,y)\in\R^d\times\R$ and for every $\theta\in\Theta^N$.
We recall that in \cite{V05} $N\big((x,y),\theta \big)$ denotes the set that contains the elements of norm $1$ in the normal cone to the constraint at $(x,y)$.
In this way, it descends that the probability measure $\mu_N\in\mathcal{P}(\Theta^N)$ prescribed by \cite[Proposition~2.1]{V05} is such that $\mathrm{supp} (\mu_N) \subset \{\theta\in \Theta^N: G\big((x^\theta_{u^*_N}(T),y^\theta_{u^*_N}(T)),\theta \big) \neq \emptyset \} = \arg\max_{\theta\in \Theta^N}a(x^\theta_{u^*_N}(T),\theta)$.
Hence, for $\theta \in \mathrm{supp} (\mu_N)$, we obtain that for a.e. $t\in[0,T]$
\begin{equation*}%\label{eq:adjoint_eq_proof}
    \begin{split}
        -\dot \lambda_{u^*_N}^\theta(t) & \in  \mathrm{co}\, \partial_x H\big(t,(x^\theta_{u^*_N}(t),y^\theta_{u^*_N}(t)), (\lambda^\theta_{u^*_N}(t),\bar\lambda^\theta_{u^*_N}(t)), u^*_N, \theta \big)\\
        & \qquad = {\textstyle \left\{ - \lambda_{u^*_N}^\theta(t)
\left( 
\frac{\partial F_0(x_{u^*_N}^\theta(t), \theta)}{\partial x}
+ \sum_{i=1}^k{u^*_{N,i}}(t) 
\frac{\partial F_i(x_{u^*_N}^\theta(t), \theta)}{\partial x} 
\right) \right\}},\\
-\dot{\bar \lambda}_{u^*_N}^\theta(t) & \in  \mathrm{co}\, \partial_y H\big(t,(x^\theta_{u^*_N}(t),y^\theta_{u^*_N}(t)), (\lambda^\theta_{u^*_N}(t),\bar\lambda^\theta_{u^*_N}(t)), u^*_N, \theta \big) = \{0\},
    \end{split}
\end{equation*}
with the final-time datum
\begin{equation*}%\label{eq:adjoint_datum_proof}
    -(\lambda_{u^*_N}^\theta(T), \bar \lambda_{u^*_N}^\theta(T) ) \in G\big((x^\theta_{u^*_N}(T),y^\theta_{u^*_N}(T)),\theta \big) = \left\{(\nabla_x a(x^\theta_{u^*_N}(T),\theta), 1)\right\}.
\end{equation*}
We observe that $\bar \lambda_{u^*_N}^\theta(t)=-1$ for every $t\in [0,T]$ and every $\theta\in\mathrm{supp}(\mu_N)$, and that $\lambda_{u^*_N}^\theta$ solves \eqref{eq:def_lambda_theta_u} for every $\theta\in\mathrm{supp}(\mu_N)$.
Finally, the maximization of the Hamiltonian in \cite[Proposition~2.1]{V05} yields
\begin{equation*}
    \begin{split}
        &\int_{\Theta^N}
        H\big(t,(x^\theta_{u^*_N}(t),y^\theta_{u^*_N}(t)), (\lambda^\theta_{u^*_N}(t),-1), u^*_N(t), \theta \big)
        \,d\mu_N(\theta)  \\
        &\qquad =
        \max_{v\in\R^k}
        \int_{\Theta^N}
        H\big(t,(x^\theta_{u^*_N}(t),y^\theta_{u^*_N}(t)), (\lambda^\theta_{u^*_N}(t),-1), v, \theta \big)
        \,d\mu_N(\theta)\\
        &\qquad 
        = \max_{v\in\R^k}
        \int_{\Theta^N}
        \left(
        \lambda^\theta_{u^*_N}(t) \cdot  (F_0(x^\theta_{u^*_N}(t),\theta) + F(x^\theta_{u^*_N}(t),\theta)v) - f(t,v)
        \right)
        \,d\mu_N(\theta).
\end{split}
\end{equation*}
Recalling the definitions of $X_N$ and $\Lambda_N$ in \eqref{eq:def_evol_ens} and in \eqref{eq:def_Lambda} respectively, we deduce the thesis.
\end{proof}

\begin{remark}
    We want to provide here an intuition for the appearance of the measure $\mu_N$ in Theorem~\ref{thm:PMP_finite}.
    In what follows, without loss of generality, we shall assume for simplicity that $J^N(u)= G\big( (x_u^\theta(T), \theta)_{\theta\in\Theta^N} \big)$, i.e., the cost is expressed in terms of the terminal states of the trajectories of the ensemble. We recall that, roughly speaking, when the dynamics and the terminal cost function $G$ are smooth, the PMP can be established by constructing the cone of the (infinitesimal) perturbations of the optimal trajectories, and by showing that the scalar product of its elements with the gradient of the terminal cost is non-negative. This approach can be pursued if e.g. $G\big( (x_u^\theta(T), \theta)_{\theta\in\Theta^N} \big) = \sum_{\theta\in\Theta^N}\nu_{\theta} a(x^\theta_{u}(T),\theta)$ with $\nu_\theta\in\R$ for every $\theta\in\Theta^N$.
    In our case, this approach is complicated by the fact that $G\big( (x_u^\theta(T), \theta)_{\theta\in\Theta^N} \big) = \max_{\theta\in\Theta^N} a(x_u^\theta(T),\theta)$ is non-smooth in the state variables. 
    In order to avoid further technicalities, let us assume in addition that the functions $x\mapsto a(x,\theta)$ are convex for every $\theta\in\Theta^N$, so that we can rely on the classical tools for subdifferential calculus.
    In this scenario, the non-negativity condition on the scalar product is generalized by requiring that \emph{there exists} an element in $\partial G\big( (x_{u^*_N}^\theta(T), \theta)_{\theta\in\Theta^N} \big)$ for which the scalar product inequality holds. Finally, we recall that
    \begin{equation*}
    \begin{split}
        \partial G\big( (x_{u^*_N}^\theta(T), \theta)_{\theta\in\Theta^N} \big) \subset &
        \Big\{
        \big( \nu_\theta \nabla_x a(x_{u^*_N}^\theta(T), \theta) \big)_{\theta\in\Theta^N}:
        \nu_\theta\geq 0, \, {\textstyle \sum_{\theta\in\Theta^N} \nu_\theta = 1
        }, \\
        &\qquad \qquad \nu_\theta = 0 \mbox{ if } 
        a(x_{u^*_N}^\theta(T), \theta) < 
        G\big( (x_{u^*_N}^\theta(T), \theta)_{\theta\in\Theta^N} \big)
        \Big\}.
    \end{split}
    \end{equation*}
    Let $(\bar \nu_\theta)_{\theta\in\Theta^N}$ be the coefficients corresponding to an element of the subdifferential that makes non-negative the scalar products mentioned above. Then, we can set $\mu_N(\{\theta\}):=\bar\nu_\theta$ for every $\theta\in\Theta^N$.
    For further details on this viewpoint, we refer to \cite{CSW24}, where the PMP for finite ensembles was studied in the framework of Machine Learning modelling.
\end{remark}

\begin{remark}\label{rmk:PMP_N}
From the proof of Theorem~\ref{thm:PMP_finite} it follows that the functional $\J^N$ does not admit any abnormal extremal.
Indeed, this due to the fact that we are considering a problem without endpoint constraints, so that the set $N\big((x,y),\theta \big)$ that contains the unitary-length elements of the normal cone to the target is always reduced to the empty set. 
%every abnormal extremal would have $\lambda^\theta_u(T)=0$, resulting in $\lambda^\theta_u(t)=0$ for every $\theta\in\Theta$ and $t\in [0,T]$.
Moreover, we observe that the growth condition \eqref{eq:growth_f} on the running cost guarantees the existence of the maximum of the weighted Hamiltonian in \eqref{eq:max_cond_N}. However, if $f$ is not strictly convex, we may have that the $\arg\max$ in \eqref{eq:max_cond_N} is not reduced to a single point.
\end{remark}

\begin{remark}
We report that the argument employed in \cite[Proposition~2.1]{V05} embraces also the case of a non-smooth dynamics, and that \emph{its proof does not require the boundedness of admissible velocities}. 
Indeed, this hypothesis is formulated later in \cite{V05} (namely \cite[Assumption~S2]{V05}), and is needed to deduce the PMP for infinite ensembles.
%For a more elementary proof of the previous result under further regularity assumptions on the control system, see \cite[Theorem~3 and Remark~5]{CSW24} with $M=1$.
\end{remark}

At this point, in view of Theorem~\ref{thm:PMP_finite}, a viable strategy to establish the Maximum Principle for infinite ensembles could be the following: first consider $\J^N\to_\Gamma\J$ (with $\J^N$ involving \emph{finite} ensembles), and then use Corollary~\ref{cor:strong_conv} to obtain the inclusion \eqref{eq:max_cond_N} in the limit. Unfortunately, not every (local) minimizer of $\J$ can be recovered as the limit of a sequence of (local) minimizers of $\J^N$, as illustrated in \cite[Remark~6.5]{Scag23}.
For this reason, given $u^*\in \U$ local minimizer for $\J$, we need to introduce an auxiliary perturbed functional $\widetilde \J:\U\to \R$ defined as follows:
\begin{equation}\label{eq:def_perturbed_fun}
\widetilde \J(u):= \sup_{\theta\in\Theta} a (X_u(T,\theta),\theta) + \int_0^T f(t,u(t))\, dt + \|u-u^*\|_{L^p}^p 
= \J(u) + \|u-u^*\|_{L^p}^p
\end{equation}
for every $u\in\U$, so that $u^*$ is an \emph{isolated} local minimizer for $\widetilde \J$.

\begin{lemma}\label{lem:isolated}
Under Hypotheses~\ref{ass:fields}-\ref{ass:running_cost} and~\ref{ass:C1_reg}, let $u^* \in\U$ be a local minimizer for the functional $\J$. Then, there exists $r>0$ such that $\widetilde J(u)>\widetilde J(u^*)$ for every $u\neq u^*$ satisfying $\|u-u^*\|_{L^p}\leq r$.
\end{lemma}
\begin{proof}
From the hypothesis, it follows that there exists $r>0$ such that $\J(u^*)\leq \J(u)$ for every $u$ such that $\|u-u^*\|_{L^p}\leq r$. Hence, recalling \eqref{eq:def_perturbed_fun}, we deduce that $\widetilde\J(u^*) = \J(u^*)\leq \J(u)< \widetilde\J(u)$ for every $u\neq u^*$ satisfying $\|u-u^*\|_{L^p}\leq r$.
\end{proof}

We now consider a sequence $(\Theta^N)_{N\in\NN}$ of \emph{finite subsets} of $\Theta$ so that $d_H(\Theta^N,\Theta)\to 0$ as $N\to\infty$, and we set $\widetilde \J^N:=\J^N+\|\cdot -u^*\|_{L^p}^p$.
Moreover, we define:
\begin{equation} \label{eq:def_ball}
\mathcal{X}:=B_r(u^*)=\{u\in\U: \|u-u^*\|_{L^p}\leq r\},
\end{equation}
where $r>0$ is provided by Lemma~\ref{lem:isolated}, and we restrict the functionals $\widetilde \J^N,\widetilde \J$ to $\mathcal \X$.
Owing to the results established in Section~\ref{sec:Gamma}, we can readily prove the following convergences.

\begin{proposition}\label{prop:G_conv_pert}
Under Hypotheses~\ref{ass:fields}-\ref{ass:running_cost} and~\ref{ass:C1_reg}, let $u^*\in\U$ be a local minimizer for the functional $\J:\U\to\R$ defined in \eqref{eq:def_fun_minimax}, and let $\widetilde \J,\widetilde \J^N:\X\to \R$ be the restrictions to $\X$ of the functionals introduced above. Then, $\widetilde \J^N\to_\Gamma \widetilde \J$ as $N\to\infty$ with respect to the $L^p$-weak topology. Moreover, if $u^*_N\in\X$ is a minimizer for the restriction $\widetilde \J^N|_\X$ for every $N\in\NN$, then $\|u^*_N -u^*\|_{L^p}\to0$ as $N\to\infty$.
\end{proposition}
\begin{proof}
The $\Gamma$-convergence descends directly from Theorem~\ref{thm:G_conv}. Moreover, since $u^*$ is the unique minimizer of $\widetilde \J$ when restricted to $\X$, it follows that any sequence of minimizers of $\widetilde \J^N|_\X$ is \emph{weakly convergent} (without extracting subsequences) to $u^*$. Finally, we observe that we can rewrite the running cost of $\widetilde \J, \widetilde \J^N$ as $\int_0^T \tilde f(t,u(t))\,dt$, where $\tilde f(t,v):=f(t,v) + |v-u^*(t)|_2^p$ for every $v\in\R^k$ and for a.e. $t\in[0,1]$.
Since $f:[0,T]\times\R^k\to\R$ is assumed to be convex in the second argument (see Hypothesis~\ref{ass:running_cost}), it turns out that $\tilde f$ is strictly convex. Therefore, Corollary~\ref{cor:strong_conv} guarantees the convergence in norm $u^*_N\to_{L^p}u^*$.
\end{proof}

We are now in position to prove the main result of this section.

\begin{theorem}\label{thm:PMP_infinite}
Under Hypotheses~\ref{ass:fields}-\ref{ass:running_cost} and~\ref{ass:C1_reg}, let $u^*\in\U$ be a local minimizer for the functional $\J:\U\to\R$ defined in \eqref{eq:def_fun_minimax}, and let $X:[0,T]\times \Theta\to\R^d$ and $\Lambda:[0,T]\times \Theta\to(\R^d)^*$ be the mappings defined, respectively, in \eqref{eq:def_evol_ens} and in \eqref{eq:def_Lambda}. 
Then, there exists a probability measure $\mu\in\mathcal{P}(\Theta)$ such that
\begin{equation}\label{eq:max_cond}
u^*(t) \in \arg \max_{\!\!\!\!\!\!
\!\!\!\!\!\! v\in \R^k} \left\{
\int_{\Theta} 
\Lambda(t,\theta)\cdot F(X(t,\theta),\theta)\cdot v \,d\mu(\theta) - f(t,v)
\right\},
\end{equation}
and
\begin{equation}\label{eq:supp}
\theta\in \mathrm{supp}(\mu) \implies
a(X(T,\theta),\theta)=
\max_{\vartheta\in\Theta} a(X(T,\vartheta),\vartheta).
\end{equation}
\end{theorem}
\begin{proof}
As above, we introduce the functionals $\widetilde J,\widetilde J^N:\U\to\R$ and the ball $\X$ defined as in \eqref{eq:def_ball}, whose radius is chosen accordingly to Lemma~\ref{lem:isolated}. Using the same notations as in Proposition~\ref{prop:G_conv_pert}, for every $N\in\NN$ we set $u^*_N\in\X$ to be a minimizer of  the restriction $\widetilde J^N|_\X$. Since $u^*_N\to_{L^p}u^*$ as $N\to\infty$, we deduce that (when $N$ is large enough) $u^*_N$ lies in the interior of $\X$, so that $u^*_N$ is actually a \emph{local minimizer for the non-restricted $\widetilde J^N$}.
Finally, the $L^p$-strong convergence ensures that, up to the extraction of a not-relabeled subsequence, $u^*_N\to u^*$ a.e. in $[0,T]$.
We set $X_N:[0,T]\times \Theta\to\R^d$ and $\Lambda_N:[0,T]\times \Theta\to(\R^d)^*$ to be the mappings defined, respectively, in \eqref{eq:def_evol_ens} and in \eqref{eq:def_Lambda} and corresponding to the control $u^*_N$.
In virtue of Theorem~\ref{thm:PMP_finite}, recalling that the running cost for the functional $\widetilde \J^N$ is given by $\tilde f(t,v):=f(t,v)+ |v-u^*(t)|_2^p$ for every $v\in\R^k$ and a.e. $t\in[0,T]$, we deduce that there exists a probability measure $\mu_N\in\mathcal{P}(\Theta^N)\subset \mathcal{P}(\Theta)$ such that 
\begin{equation}\label{eq:max_cond_subd}
0\in \partial_{u^*_N(t)}\, g^t_N(\cdot)
\end{equation}
for every $t\in[0,T]\setminus Z_N$, where $Z_N\subset[0,T]$ is a set of Lebesgue measure zero, and
\begin{equation*}
g_N^t(v):= \tilde f(t,v) - 
\int_\Theta
\Lambda_N(t,\theta)\cdot F(X_N(t,\theta),\theta)
 \,d\mu_N(\theta)\cdot v.
\end{equation*} 
Then, recalling the definition of $\tilde f$, we can rewrite \eqref{eq:max_cond_subd} as
\begin{equation}\label{eq:max_cond_subd_2}
\int_\Theta
\Lambda_N(t,\theta)\cdot F(X_N(t,\theta),\theta)
 \,d\mu_N(\theta)
 \in \partial_{u^*_N(t)}\, \tilde f(t,\cdot)\iff 
q_N(t)
 \in \partial_{u^*_N(t)}\, f(t,\cdot)
\end{equation}
for every $t\in[0,T]\setminus Z_N$,
where $q_N:[0,T]\to(\R^d)^*$ is defined as:
\begin{equation} \label{eq:def_q_N}
q_N(t):= \int_\Theta
\Lambda_N(t,\theta)\cdot F(X_N(t,\theta),\theta)
 \,d\mu_N(\theta)
 - p|u^*_N(t)-u^*(t)|_2^{p-1}\frac{u^*_N(t)-u^*(t)}{|u^*_N(t)-u^*(t)|_2}.
\end{equation}
We report that in the passages above we applied the \textit{sum rule for subdifferentials} 
(see, e.g., \cite[Theorem~23.8]{R97}) to $\tilde f(t,\cdot)$, which is the sum of a (possibly non-smooth) convex function $f(t,\cdot)$ with a regular and strictly convex one.\\
Before passing to the limit in \eqref{eq:max_cond_subd_2}, we extract from the sequence of probability measures $(\mu_N)_{N\geq1}\subset\mathcal{P}(\Theta)$ a (not relabeled) sub\-sequence that admits a weak limit $\mu\in\mathcal{P}(\Theta)$. Here, we use the weak-$*$ topology induced by the pairing with continuous functions $C^0(\Theta,\R)$ (see \cite[Definition~3.5.1]{C05}), while the existence of a convergent sub\-sequence descends from Prokhorov Theorem \cite[Theorem~3.5.13]{C05}.
Moreover, combining Proposition~\ref{prop:unif_conv_map_Lambda} with  Proposition~\ref{prop:unif_conv_map_X} and Hypothesis~\ref{ass:fields}, it follows that
\begin{equation*}
\lim_{N\to\infty}
\sup_{t\in[0,T],\theta\in\Theta} \left|
\Lambda_N(t,\theta)\cdot F(X_N(t,\theta),\theta) -
\Lambda(t,\theta)\cdot F(X(t,\theta),\theta)
\right|,
\end{equation*}
where $X:[0,T]\times \Theta\to\R^d$ and $\Lambda:[0,T]\times \Theta\to(\R^d)^*$ are the maps corresponding to $u^*$. Therefore, for every $t\in[0,T]$, we have that
\begin{equation}\label{eq:conv_int_q_N}
\lim_{N\to\infty} \int_\Theta
\Lambda_N(t,\theta)\cdot F(X_N(t,\theta),\theta)
 \,d\mu_N(\theta) = 
 \int_\Theta
\Lambda(t,\theta)\cdot F(X(t,\theta),\theta)
 \,d\mu(\theta).
\end{equation}
Moreover, we denote with $Z_\infty\subset[0,T]$ the set of null Lebesgue measure where the pointwise convergence $u_N^*\to u^*$ fails, and we define $Z:=\bigcup_{N\geq1}Z_N\cup Z_\infty$.
Hence, for every $t\in[0,T]\setminus Z$, we have that
\begin{equation}\label{eq:conv_frac_q_N}
\lim_{N\to\infty} |u^*_N(t)-u^*(t)|_2^{p-1}\frac{u^*_N(t)-u^*(t)}{|u^*_N(t)-u^*(t)|_2} =0.
\end{equation}
Combining \eqref{eq:conv_int_q_N} and \eqref{eq:conv_frac_q_N} in \eqref{eq:def_q_N}, we obtain 
\begin{equation}\label{eq:conv_q_N}
\lim_{N\to\infty} q_N(t) = 
 \int_\Theta
\Lambda(t,\theta)\cdot F(X(t,\theta),\theta)
 \,d\mu(\theta) =:q(t)
\end{equation}
for every $t\in[0,T]\setminus Z$.
We are now in position to pass to the limit into \eqref{eq:max_cond_subd_2}: in virtue of the \textit{continuity of subdifferentials} \cite[Theorem~24.4]{R97}, since $q_N(t)\to q(t)$ and $u^*_N(t)\to u^*(t)$ as $N\to\infty$ for $t\in[0,T]\setminus Z$, we conclude that
\begin{equation*}
q(t) = \int_\Theta
\Lambda(t,\theta)\cdot F(X(t,\theta),\theta)
 \,d\mu(\theta) \in \partial_{u^*(t)}\, f(t,\cdot),
\end{equation*}
which is equivalent to \eqref{eq:max_cond}.
We are left to show that \eqref{eq:supp} holds.
To see that, we first observe that, if $\theta\in\mathrm{supp}(\mu)$, then there exists an increasing sequence $(N_j)_{j\geq 1}$ such that we may find $\theta_{N_j}\in\mathrm{supp}(\mu_{N_j})$ and $\theta_{N_j}\to \theta$ as $j\to\infty$. Then, in virtue of Remark~\ref{rmk:cluster_maxim_2}, we deduce that $\theta\in\arg\max_{\vartheta\in\Theta} a(X(T,\vartheta),\vartheta)$.
\end{proof}

\begin{remark}\label{rmk:measure_meaning}
The measure $\mu\in\mathcal{P}(\Theta)$ that appears in \eqref{eq:max_cond} is in charge of highlighting the values of the parameter where worst cases are attained, and of ``ranking'' them with a weight.
Moreover, it is interesting to observe that Theorem~\ref{thm:PMP_infinite} can be restated by saying that any local minimizer $u^*\in \U$ of the minimax functional $\J$ is as well a normal extremal for the following \emph{weighted} ensemble optimal control problem:
\begin{equation*}
    \J_\mu(u):= 
    \int_\Theta a(X_u(T,\theta),\theta)\, d\mu(\theta) 
    + \int_0^T f(t,u(t))\, dt \to \min,
\end{equation*}
see \cite[Theorem~6.3]{Scag23}.

\end{remark}

\section{Application to quantum control and numerical examples}
In this section, we apply the reduction argument provided by $\Gamma$-convergence to the problem of controlling a qubit whose resonance frequency is affected by uncertainty. More precisely, we consider the two-level system described by the Schr\"odinger equation
\begin{equation}\label{eq:qubit_sys}
i\dot\psi_u^\alpha(t) = \big(H_0^\alpha + H_1 u(t) \big)\psi_u^\alpha(t),
 \quad \psi_u^\alpha(0) =
\left(
\begin{matrix}
0\\
1
\end{matrix}
\right),
\end{equation}
with
\begin{equation*}
H_0^\alpha := \left(
\begin{matrix}
E+\alpha & 0 \\
0 & -E-\alpha
\end{matrix}
\right),
\quad
H_1 := \left(
\begin{matrix}
0 & 1 \\
1 & 0
\end{matrix}
\right),
\end{equation*}
where $\psi_u^\alpha:[0,T]\to\mathbb{C}^2$ is the wave function, $E>0$, and $\alpha\in[\alpha_0,\alpha_1]$ is the unknown quantity that parametrizes the ensemble. Finally, $u:[0,T]\to\R$ is the real-valued signal used to control the system.
In the recent paper \cite{RABS}, the authors considered the problem of designing a control such that, at the terminal instant $T$, the elements of the ensemble \eqref{eq:qubit_sys} are close to the state $\psi^{\mathrm{tar}}:=(1,0)^T$ (up to a phase, possibly depending on $\alpha$). In other words, the goal was to make $\inf_{\alpha\in [\alpha_0,\alpha_1]} \big|\langle \psi^{\mathrm{tar}} | \psi^{\alpha}_u(T) \rangle\big|$ as large as possible, and their main achievement consisted in a uniform controllability result for \eqref{eq:qubit_sys} (see \cite[Theorem~3]{RABS}), together with the explicit construction of families of controls to accomplish the task (see \cite[Remark~5]{RABS}). We summarize these results below, and we refer the reader to the original paper for the complete and more general statements.

\begin{theorem}[{[29, Theorem~3]}] \label{thm:qubit_control}
 Let us assume that $\alpha_0<0<\alpha_1$, and that $3(E+\alpha_0)\geq E+\alpha_1$. For every $\e_1,\e_2>0$, let us consider the control $u_{\e_1,\e_2}:[0,1/(\e_1\e_2)]\to \R$ defined as:
\begin{equation}\label{eq:control_qubit}
u_{\e_1,\e_2}(t):=
2\e_1\big(1-\cos(2\pi\e_1\e_2 t)\big)
\cos\left(
2Et
+\frac{(\alpha_0-\alpha_1)\sin(\pi\e_1\e_2 t)}{\pi\e_1\e_2}
+ (\alpha_0 + \alpha_1)t
\right),
\end{equation}
and let us denote with $\psi^\alpha_{\e_1,\e_2}$ the solution of \eqref{eq:qubit_sys} corresponding to $u_{\e_1,\e_2}$.
Then, for every $K_0\in\mathbb{N}$, for every compact $I'\subset (\alpha_0,\alpha_1)$, there exists $C_{K_0}>0$ and $\eta>0$ such that, for every $\alpha\in I'$ and $\e_1,\e_2\in (0,\eta)$, we have
\begin{equation*}
\bigg|
\psi^\alpha_{\e_1,\e_2}\left( \frac{1}{\e_1\e_2} \right)- (e^{i\zeta},0)^T
\bigg| < C_{K_0}\max\left\{
\frac{\e_2}{\e_1}, \frac{\e_1^{K_0-1}}{\e_2}
\right\}
\end{equation*}
for some $\zeta\in[0,2\pi]$ (possibly depending on $\alpha$).
\end{theorem}

In this section, inspired by the previous result, we study numerically the minimax problem:
\begin{equation}\label{eq:def_qubit_fun}
\J(u):= \max_{\alpha\in[\alpha_0,\alpha_1]} \left(
1- \big|\langle \psi^{\mathrm{tar}} | \psi^{\alpha}_u(T) \rangle\big|^2
\right) + \gamma \| u \|_{L^2}^2\to \min,
\end{equation}
and we compare the outcomes with the ones corresponding to controls prescribed by \eqref{eq:control_qubit}, with $\e_1,\e_2$ satisfying $T=1/(\e_1\e_2)$.
We observe that tuning the parameter $\gamma>0$ is an essential phase in this situation, where we first have to define a suitable optimal control problem and then numerically solve it. Indeed, on the one hand, high values of $\gamma$ result in high convexity and coercivity of the functional $\J$, which in turn reflects good numerical tractability.
On the other hand, for large $\gamma$, we may expect that the minimizers of \eqref{eq:def_qubit_fun} tend to perform poorly on the final-time minimax term.
Hence, as is often the case with Tikhonov-regularized minimization problems, properly setting $\gamma$ traduces in finding a trade-off between the numerical properties of the problem we want to solve and how the optimal controls behave on the minimax part.
In practice, a viable option consists in adaptively adjusting $\gamma$ by starting from a reasonably large value, and then progressively diminishing it until the minimax term evaluated on the numerical solutions matches a desired level.  
In virtue of Theorem~\ref{thm:G_conv}, we substitute the infinite set of parameters $I=[\alpha_0,\alpha_1]$ with a regular discretization $I^N$, and we address the reduced problem:
\begin{equation}\label{eq:def_qubit_fun_N}
\J^N(u):= \max_{\alpha\in I^N} \left(
1- \big|\langle \psi^{\mathrm{tar}} | \psi^{\alpha}_u(T) \rangle\big|^2
\right) + \gamma \| u \|_{L^2}^2\to \min.
\end{equation}
We employed a subgradient-like scheme for the minimization of $\J^N$ (outlined in Algorithm~\ref{alg:iter_PMP_subgrad}) consisting of the following main steps:
\begin{enumerate}
\item Given a control $u$, compute $\psi^\alpha_u$ for every $\alpha\in I^N$, and find a $\bar \alpha$ where the maximum in \eqref{eq:def_qubit_fun_N} is attained.
For the current choice of the control, such a value $\bar \alpha$ represents the \emph{worst-case}, i.e. the value of the parameter for which the end-point cost is the largest. 
\item Consider  $\J^{\bar \alpha}(u):= 
1- \big|\langle \psi^{\mathrm{tar}} | \psi^{\bar \alpha}_u(T) \rangle\big|^2
 + \gamma \| u \|_{L^2}^2$ and update $u$ by performing \emph{one iteration} of a PMP-based indirect method on $\J^{\bar \alpha}$. Then, return to Step~1.
\end{enumerate}
The statement of the PMP used for the implementation and related to the functional $\J^{\bar \alpha}$ (i.e., to a \emph{single system} of the ensemble) is detailed in \cite[Example~3]{BSS21}. 
We notice the peculiar aspect of the Hamiltonian of the system (appearing also at line~7 of Algorithm~\ref{alg:iter_PMP_subgrad}), where the imaginary part relates to the fact that the state variable $\psi_u^\alpha$ of the quantum system takes values in $\mathbb{C}^2$.

\begin{algorithm}
\KwData{ 
$I^N\subset I$
subset of parameters, $T>0$ (time horizon).\\
{\bf Algorithm setting:} 
$\tau_0 \gg 0 $, $\max_{\mathrm{iter}}\geq 1$,
$u\in L^2([0,T],\R)$ (initial guess).}

$\tau\gets \tau_0$\;
\For(\tcp*[f]{Iterations of the algorithm}){$n=1,\ldots,\max_{\mathrm{iter}}$ }{

	$\forall\, \alpha\in I^N$, compute the trajectory $\psi^\alpha_u$ solving \eqref{eq:qubit_sys} with the control $u$\;
Find $\bar \alpha\in I^N$ such that $\bar \alpha \in \arg \max_{\alpha\in I^N} \left(
1- \big|\langle \psi^{\mathrm{tar}} | \psi^{\alpha}_u(T) \rangle\big|^2 \right)$\;
		Compute the adjoint trajectory $\chi^{\bar \alpha}_u$ solving bacward the adjoint system of \eqref{eq:qubit_sys}\;
	Update the control by maximizing the augmented Hamiltonian: $u(t) \gets
\arg \max_{ v\in \R}
\left\{ \Im\big( \langle \chi^{\bar \alpha}_u(t) | H_1v |\psi^{\bar \alpha}_u(t) \rangle \big) 
- \frac{\gamma}{2}|v|^2
-\frac{\tau}{2}|v-u(t)|^2
\right\}$\;
$\tau\gets \tau_0 + n$\;
	}
\caption{Iterative Maximum Principle on the worst-case (subgradient-like)}
\label{alg:iter_PMP_subgrad}
\end{algorithm}

\begin{remark}\label{rmk:subgradient}
We refer to Algorithm~\ref{alg:iter_PMP_subgrad} as a `subgradient-like' scheme in analogy to the finite dimensional convex case. Namely, let $g_1,\ldots,g_N:\R^d\to\R$ be $C^1$-regular convex functions, and let us define $g(x):=\max_{i=1,\ldots,N}g_i(x)$ for every $x\in\R^d$. Then, at every point $x$, we have that $\partial g(x) = \mathrm{co}\{\nabla g_j(x): g_j(x)=g(x)\}$, so that the iterations $x_{n+1}=x_n - h_n\nabla g_{\hat i}(x_n)$ (with $\hat i$ such that $g_{\hat i}(x_n)=g(x_n)$ and $h_n>0$) result in a \emph{subgradient descent scheme}. 
We recall that a sufficient condition for having convergence requires the step-size to satisfy $h_n\to 0$ and $\sum_{n=1}^\infty h_n=\infty$ (see e.g. \cite[Theorem~32]{Shor98}).
\end{remark}

\begin{remark}\label{rmk:costate}
    In Line~6 of Algorithm~\ref{alg:iter_PMP_subgrad}, given the current guess $u$ and the forward path $ \psi^{\bar \alpha}_u$, we compute the adjoint trajectory by solving backward in time 
    \begin{equation}\label{eq:qubit_sys_adj}
    i\dot\chi_u^{\bar \alpha}(t) = \big(H_0^{\bar \alpha} + H_1 u(t) \big)\chi_u^{\bar \alpha}(t),
    \quad 
    \chi_u^{\bar \alpha}(T) = -2\langle \psi^{\mathrm{tar}} | \psi^{\bar \alpha}_u(T) \rangle \psi^{\mathrm{tar}},
\end{equation}
where we understand for convenience $\chi_u^{\bar \alpha}$ as a \emph{column} vector in $\mathbb{C}^2$. We refer to \cite[Example~3]{BSS21} for further details.
\end{remark}

\begin{remark}\label{rmk:stepsize}
In Line~7 of Algorithm~\ref{alg:iter_PMP_subgrad} we perform the update of the control by maximizing an \emph{augmented} Hamiltonian (see \cite{SS80}). Indeed, the extra term $-\frac\tau2|v-u(t)|^2_2$ penalizes large deviations from the control computed at the previous iteration, and the positive quantity $\tau$ plays the role of the \emph{inverse} of the descent step-size.
The fact that $\tau$ increases at each iteration (Line~8) is motivated by what is explained at the end of Remark~\ref{rmk:subgradient}.
\end{remark}

\subsection*{Experimental setting and results}
We adopted the same parameters chosen in \cite{RABS} for the illustrations, i.e., $E=1$, $\alpha_0=-0.5$ and $\alpha_1=0.5$, while in \eqref{eq:def_qubit_fun}-\eqref{eq:def_qubit_fun_N} we set $\gamma=2^{-4}$.
We approximated $I=[-0.5,0.5]$ with $I^N$, consisting of $N=101$ equi-spaced points at distance $0.01$.
We considered as time horizon $T=20$ and $T=50$, and in both cases we used a constant time-step discretization $\Delta t=2^{-5}$, and we took controls $u$ piecewise constant on each subinterval $[k\Delta t,(k+1)\Delta t)$.
We stress the fact that on each subinterval we evolved \eqref{eq:control_qubit} (and the adjoint equation as well) by computing the \emph{exact exponential} of the matrix $H_0^\alpha + H_1 u$.
We used as initial guess for Algorithm~\ref{alg:iter_PMP_subgrad} a control obtained after $400$ iterations of a PMP-based indirect scheme applied to the \emph{averaged} problem
\begin{equation} \label{eq:def_qubit_fun_av}
\J^{\mathrm{av}}(u) := \frac1N \sum_{\alpha\in I^N}
\left(
1- \big|\langle \psi^{\mathrm{tar}} | \psi^{\alpha}_u(T) \rangle\big|^2
\right) + \gamma \| u \|_{L^2}^2\to \min,
\end{equation}
where every element of $I^N$ was assigned the same weight $1/N$. In this way, we can understand the final output of Algorithm~\ref{alg:iter_PMP_subgrad} as an adjustment of an averaged ensemble optimal control problem. Finally, we run Algorithm~\ref{alg:iter_PMP_subgrad} for $1000$ iterations with $\tau_0 = 8$.

\begin{remark}
The scheme used for $\J^{\mathrm{av}}$ is extremely similar to Algorithm~\ref{alg:iter_PMP_subgrad}. The main differences are that the adjoint trajectories $\chi^\alpha_u$ needs to be computed for every $\alpha\in I^N$, and that the maximization of the averaged (augmented) Hamiltonian reads as 
$u(t) \gets
\arg \max_{ v\in \R}
\left\{ 
\frac1N \sum_{\alpha \in I^N}
\Im\big( \langle \chi^{ \alpha}_u(t) | H_1v |\psi^{ \alpha}_u(t) \rangle \big) 
- \gamma|v|^2_2
-\frac{\tau}{2}|v-u(t)|^2_2
\right\}$.
%follows:
%\begin{equation*}
%u(t) \gets
%\arg \max_{ v\in \R}
%\left\{ 
%\frac1N \sum_{\alpha \in I^N}
%\Im\big( \langle \chi^{ \alpha}_u(t) | H_1v |\psi^{ \alpha}_u(t) \rangle \big) 
%- \gamma|v|^2_2
%-\frac{\tau}{2}|v-u(t)|^2_2
%\right\}.
%\end{equation*}
Finally, we used $\tau= 4$ at every iteration.
\end{remark}

The results are reported in Figure~\ref{fig:results}.
On the one hand, we observe that the controls constructed according to \eqref{eq:control_qubit} (blue) and by minimizing the averaged cost \eqref{eq:def_qubit_fun_av} (red) exhibit a performance deterioration when the parameter $\alpha$ is close to the extreme of the interval $[\alpha_0, \alpha_1]$, as we can see in the graphs on the left hand-side column. 
On the other hand, when considering the solution of \eqref{eq:def_qubit_fun_N} computed with Algorithm~\ref{alg:iter_PMP_subgrad} (yellow), we observe that the quantity $\big|\langle \psi^{\mathrm{tar}} | \psi^{\alpha}_u(T) \rangle\big|$ shows less pronounced oscillations as $\alpha$ ranges in its domain.
Moreover, we notice that the controls computed by minimization of \eqref{eq:def_qubit_fun_N} and \eqref{eq:def_qubit_fun_av} have lower energy (in terms of the $L^2$-norm) than the ones obtained through \eqref{eq:control_qubit}.\\
Finally, we report in Table~\ref{table:table_20} and Table~\ref{table:table_50} the comparison of the performances of approximated optimal controls, obtained by constructing $I^N$ with $N=26, 51, 101$ equi-spaced points, for $T=20$ and $T=50$, respectively. 
The computation of $\max_\alpha$ (the quantity of interest) and $\min_\alpha$ is carried out by evolving a \emph{discrete test ensemble} consisting of $1001$ equi-spaced points on $I$.
We observe that, for $T=20$, the values of $\max_\alpha$ are relatively close (the largest is $1.40\%$ larger than the smallest).
For $T=50$, we actually observe that the best performing control in terms of $\max_\alpha$ is the one corresponding to $N=101$, and the relative distance between the smallest and the largest is $13.0\%$.

\begin{table}
\scriptsize
\centering
$T=20$
\vspace{2pt}

\begin{tabular}{|c c c c|}
\hline
$N$ & $\max_{\alpha} \left(1- \big|\langle \psi^{\mathrm{tar}} | \psi^{\alpha}_u(T) \rangle\big|\right)$ & $\min_{\alpha} \left(1- \big|\langle \psi^{\mathrm{tar}} | \psi^{\alpha}_u(T) \rangle\big|\right)$ & $\|u\|_{L^2}^2$\\ [0.6ex]
\hline
$26$ & $0.0498$ & $0.0283$ & $1.7162$\\
\hline
$51$ & $0.0500$ & $0.0270$ & $1.7098$\\
\hline
$101$ & $0.0505$ & $0.0264$ & $1.7058$ \\
\hline
\end{tabular}
\vspace{3pt}
\caption{}\label{table:table_20}
\end{table}

\begin{table}
\scriptsize
\centering
$T=50$
\vspace{2pt}

\begin{tabular}{|c c c c|}
\hline
$N$ & $\max_{\alpha} \left(1- \big|\langle \psi^{\mathrm{tar}} | \psi^{\alpha}_u(T) \rangle\big|\right)$ & $\min_{\alpha} \left(1- \big|\langle \psi^{\mathrm{tar}} | \psi^{\alpha}_u(T) \rangle\big|\right)$ & $\|u\|_{L^2}^2$\\ [0.6ex]
\hline
$26$ & $0.0476$ & $0.0191$ & $1.7622$\\
\hline
$51$ & $0.0449$ & $0.0128$ & $1.7595$\\
\hline
$101$ & $0.0421$ & $0.0164$ & $1.7652$ \\
\hline
\end{tabular}
\vspace{3pt}
\caption{}\label{table:table_50}
\end{table}

\begin{figure}
\centering
\includegraphics[scale=0.3]{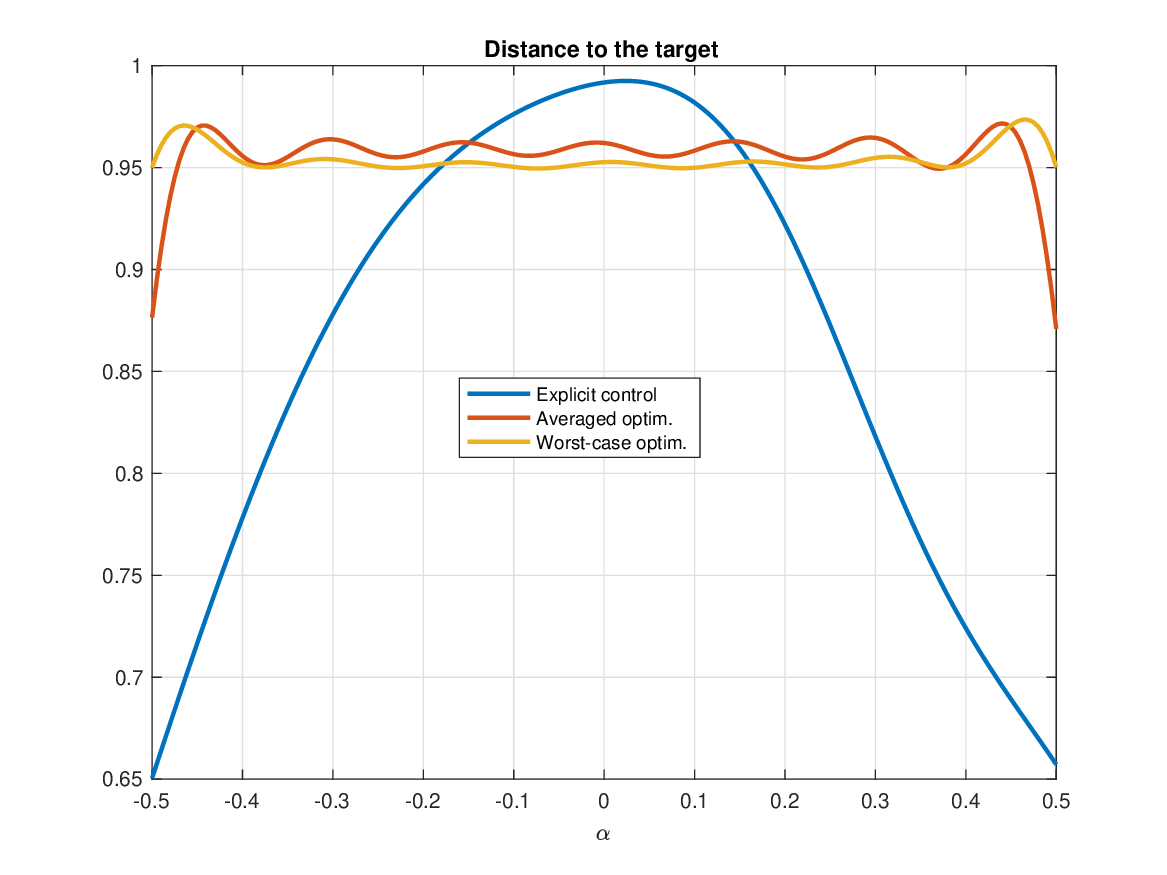}
\includegraphics[scale=0.3]{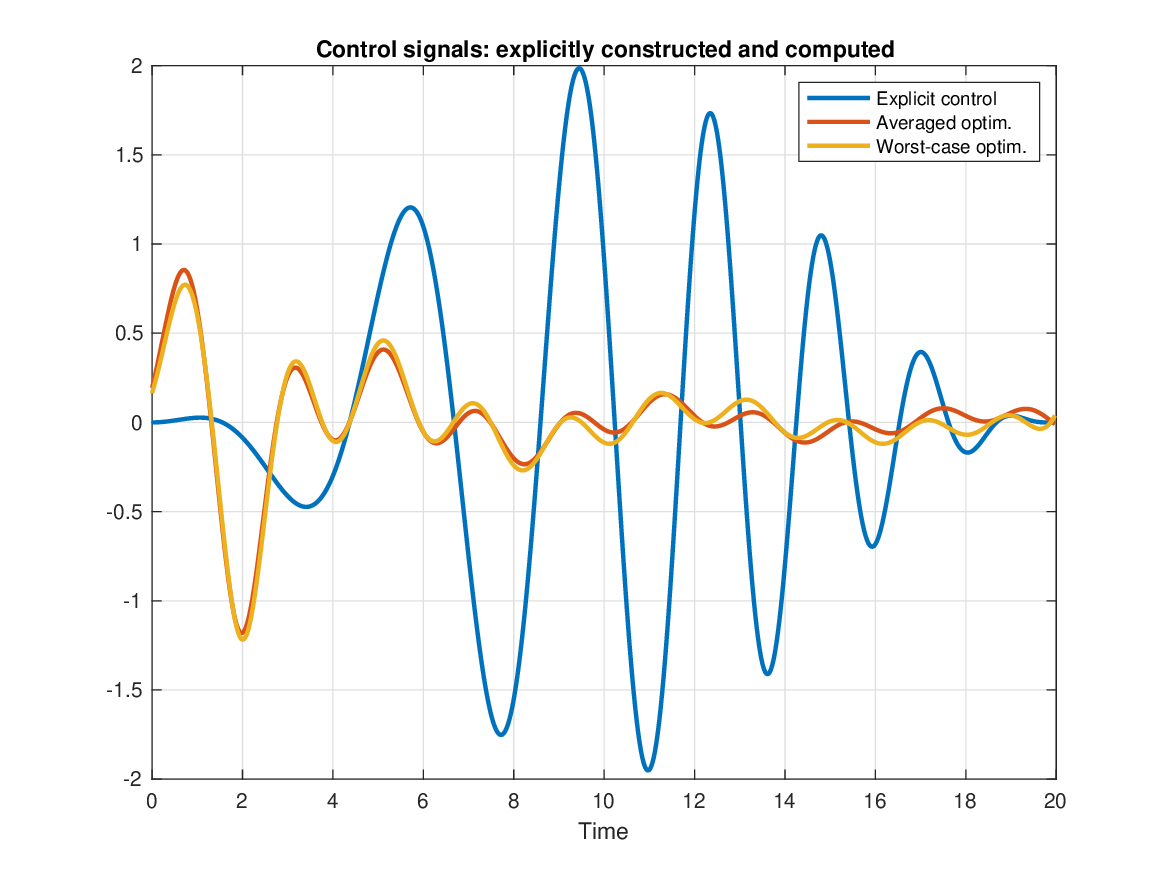}\\
 \includegraphics[scale=0.3]{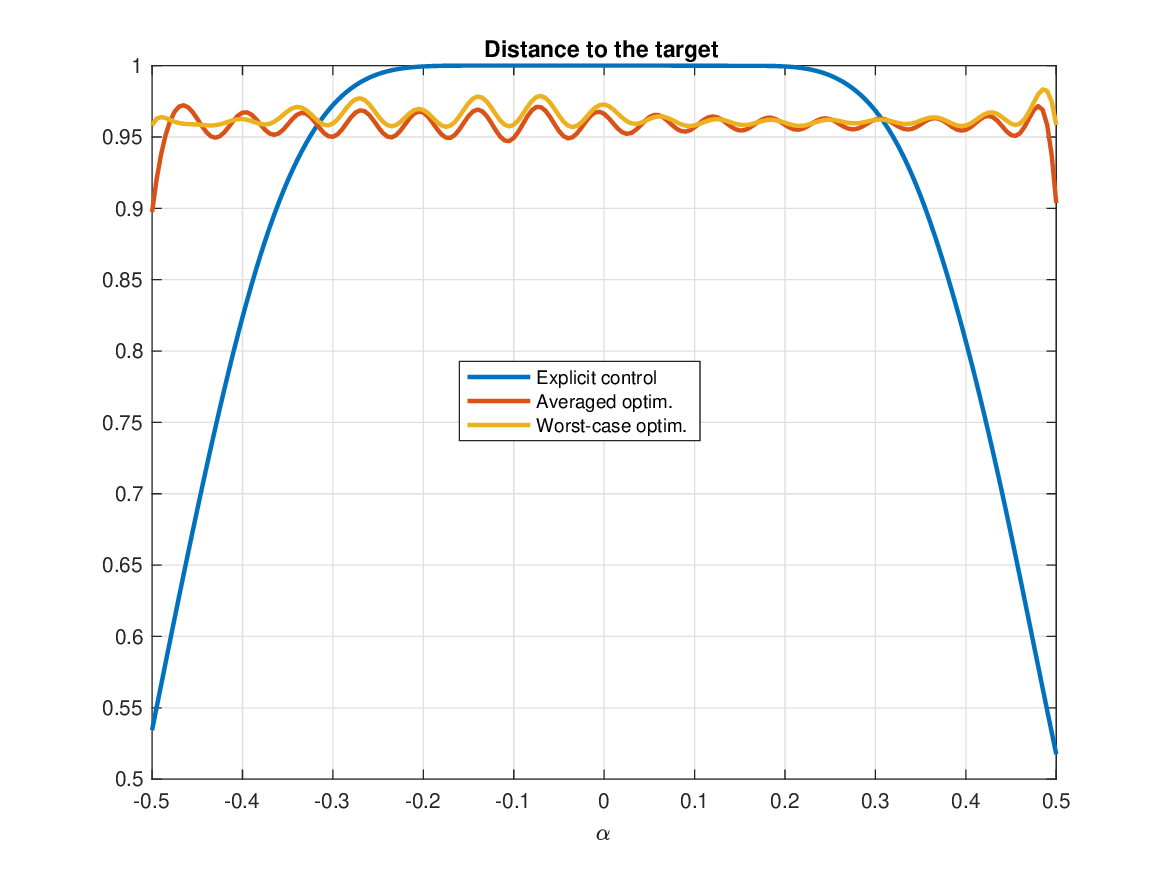}
 \includegraphics[scale=0.3]{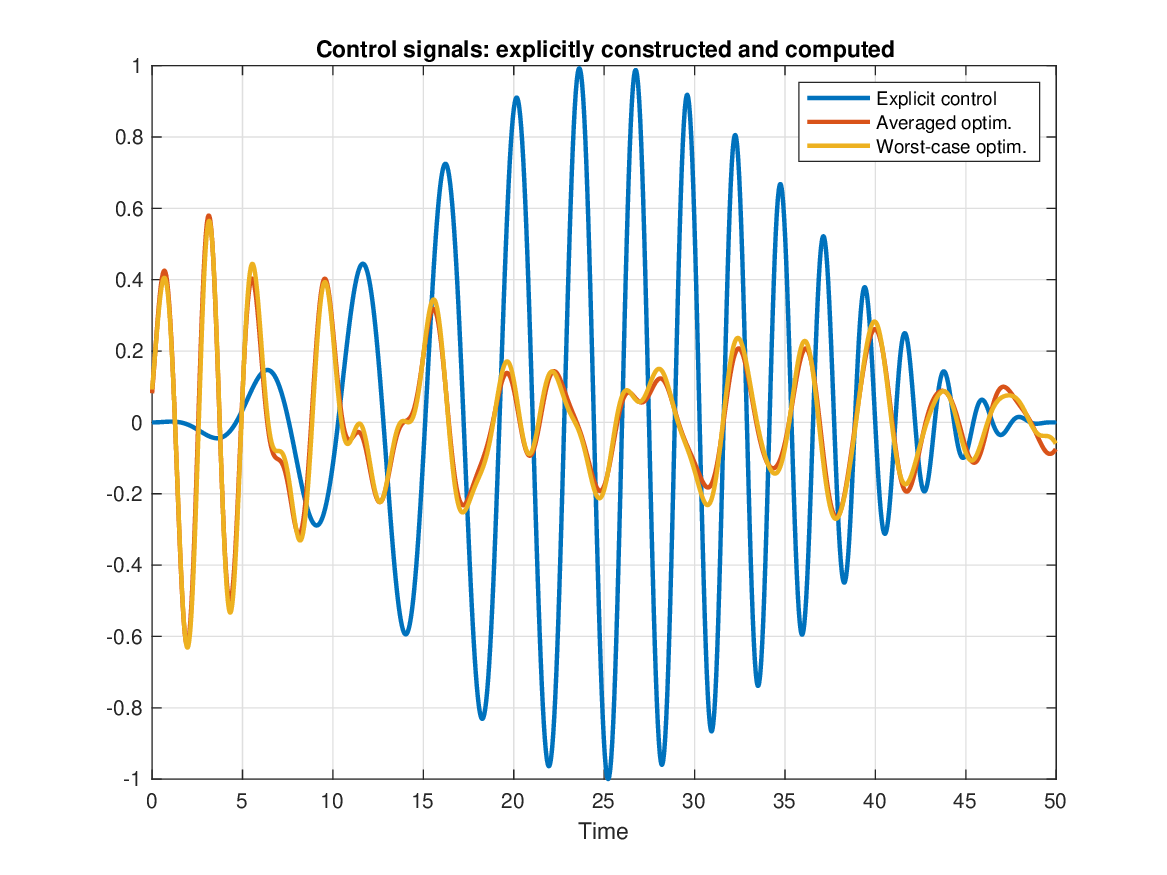}
\caption{We compared the results obtained by minimizing numerically \eqref{eq:def_qubit_fun_av} (red), \eqref{eq:def_qubit_fun_N} (yellow), and by considering the control prescribed by \eqref{eq:control_qubit} (blue). The left plots represent $\big|\langle \psi^{\mathrm{tar}} | \psi^{\alpha}_u(T) \rangle\big|$ (desirably, as close as possible to $1$) as $\alpha$ ranges in $[-0.5,0.5]$, while at the right-hand side we reported the graphs of the control used to drive the ensemble. In the pictures above, we set $T=20$, and $\e_1=0.5, \e_2=0.1$, while below we set $T=50$, $\e_1=0.25, \e_2=0.08$.}
\label{fig:results}
\end{figure}

\section*{Conclusions}

In this paper, we considered minimax optimal control problems involving ensembles of control-affine systems.
Using tools from Calculus of Variations, we proved the existence of solutions for the optimization problem, and we established their stability when the sets that parametrize the ensembles are converging in the Hausdorff distance.
The latter property (formulated as a $\Gamma$-convergence result) enabled us to establish the PMP when the ensemble consists of infinitely many systems. \\
Finally, we proposed an empirical scheme for the numerical resolution of this kind of problems, and we tested it in the framework described in \cite{RABS}.

\subsection*{Acknowledgments}
The Author wants to thank Dr. Nicolas Augier for the fruitful discussions that motivated the investigations in the present paper.
Moreover, we want to thank two anonymous Referees for their constructive comments and observations.
The Author acknowledges partial support from INdAM-GNAMPA.

\vspace{1cm}

\noindent

\begin{small}
\noindent
(A. Scagliotti).

\vspace{5pt} 
\noindent
\textsc{``School of Computation, Information and Technology'',\\ Technical University of Munich, Garching b. M\"unchen, Germany.}
\vspace{3pt}

\noindent
\textsc{Munich Center for Machine Learning (MCML), Germany.}

\vspace{5pt} 
\noindent
\textit{Email address:} \texttt{scag -at- ma.tum.de}
\end{small}


\begin{thebibliography}{99}



\bibitem{ABS16}
A. Agrachev, Y. Baryshnikov, A. Sarychev. 
{Ensemble controllability by Lie algebraic
methods.}
{\it ESAIM: Control. Optim. Calc. Var.},
22:921-938 (2016).
doi: 10.1051/cocv/2016029

\bibitem{AL24}
A. Agrachev, C. Letrouit.
Generic controllability of equivariant systems and applications to particle systems and neural networks. 
\textit{arXiv preprint}, arXiv:2404.08289 (2024).

%\bibitem{AS} 
%A. Agrachev, Yu. Sachkov.
%{\it Control Theory from the Geometric Viewpoint.}
%Encyclopaedia of Mathematical Sciences,
%Springer-Verlag Berlin Heidelberg (2004).
%doi: 10.1007/978-3-662-06404-7

\bibitem{AS1}
A. Agrachev, A. Sarychev.
Control in the spaces of ensembles of points.
\textit{SIAM J. Control Optim.}, 58: 1579-1596 (2020)  
doi: 10.1137/19M1273049

\bibitem{AS2}
A. Agrachev, A. Sarychev.
Control on the manifolds of mappings with a view to the deep learning.
\textit{J. Dyn. Control Syst.}, 28: 989-1008 (2022).
doi: 0.1007/s10883-021-09561-2


\bibitem{AuBoSi18}
N. Augier, U. Boscain, M. Sigalotti. 
Adiabatic ensemble control of a continuum of quantum 
systems. 
{\it SIAM J. Control Optim.},
56(6):4045-4068 (2018).
doi: 10.1137/17M1140327



\bibitem{BCR10}
K. Beauchard, J.-M. Coron, P. Rouchon.
{Controllability issues for continuous-spectrum
systems and ensemble controllability of Bloch equations}
{\it Comm. Math. Phys.}, 296(2):525-557
(2010).
doi: 10.1007/s00220-010-1008-9



\bibitem{BST15}
M. Belhadj, J. Salomon, G. Turinici.
{Ensemble controllability and discrimination of
perturbed bilinear control systems on connected, 
simple, compact Lie groups.}
{\it Eur. J. Control}, 22:23-29 (2015).
doi: 10.1016/j.ejcon.2014.12.003



\bibitem{BK19}
P. Bettiol, N. Khalil.
Necessary optimality conditions for average cost 
minimization problems.
\textit{Discete Contin. Dyn. Syst. - B},
24(5): 2093-2124 (2019).
doi: 10.3934/dcdsb.2019086 

\bibitem{BK24}
P. Bettiol, N. Khalil.
Average cost minimization problems subject to state constraints.
\textit{SIAM J. Control Optim.}, 62(3):1884-1907 (2024).
doi: 10.1137/23M1558124

\bibitem{BSS21}
U. Boscain, M. Sigalotti, D. Sugny.
Introduction to the Pontryagin Maximum Principle for Quantum Optimal Control.
\textit{PRX Quantum}, 2(3):1-31
(2021).
doi: 10.1103/PRXQuantum.2.030203

\bibitem{B11}
H. Brezis.
{\it Functional Analysis, Sobolev Spaces and 
Partial Differential Equations.}
Universitext, 
Springer New York NY (2011).
doi: 10.1007/978-0-387-70914-7







\bibitem{ChiGau18}
F.C. Chittaro, J.P. Gauthier.
{Asymptotic ensemble stabilizability of the Bloch 
equation.}
{\it Sys. Control Lett.},
113:36-44 (2018).
doi: 10.1016/j.sysconle.2018.01.008


\bibitem{C05}
E. \c{C}inlar. 
{\it Probability and Stochastics.}
Graduate Texts in Mathematics,
Springer-Verlag, New York (2010).
doi: 10.1007/978-0-387-87859-1


\bibitem{CFS24}
C. Cipriani, M. Fornasier, A. Scagliotti.
From NeurODEs to AutoencODEs: a mean-field control framework for width-varying neural networks.
\textit{Eur. J. Appl. Math.}, 1-43 (2024).
doi: 10.1017/S0956792524000032

\bibitem{CSW24}
C. Cipriani, A. Scagliotti, T. W\"oherer.
A minimax optimal control approach for robust neural ODEs.
\textit{2024 European Control Conference (ECC24)}, to appear (2024).

%\bibitem{CL}
%F. Chernousko, A. Lyubushin. 
%Method of successive approximations for solution of 
%optimal control problems. 
%{\it Opt. Control 
%Appl. Methods}, 3(2):101-114 (1982).


\bibitem{D93} G. Dal Maso. 
\textit{An Introduction to $\Gamma$-convergence.}
Progress in nonlinear differential equations and 
their applications,
Birkh\"auser Boston MA (1993).


\bibitem{Dan22}
B. Danhane, J. Loh\'eac, M. Jungers. 
Conditions for uniform ensemble output controllability, and obstruction to uniform ensemble controllability.
\textit{Math. Control Rel. Fields},
(2023). doi: 10.3934/mcrf.2023036


\bibitem{DS21}
G. Dirr, M. Sch\"{o}nlein.
{Uniform and $L^q$-ensemble reachability of parameter-dependent linear systems}
{\it J. Differ. Eq.}, 
283:216-262 (2021).
doi: 10.1016/j.jde.2021.02.032


\bibitem{FL07}
I. Fonseca, G. Leoni.
\textit{Modern Methods in the Calculus of Variations: $L^p$ spaces.}
Springer Monographs in Mathematics, Springer New York NY
(2007).
doi: 10.1007/978-0-387-69006-3


\bibitem{GZ22}
B. Geshkovski, E. Zuazua. 
Turnpike in optimal control of PDEs, ResNets, and beyond.
\textit{Acta Num.}, 31:135-263 (2022).
doi: 10.1017/S0962492922000046


\bibitem{GLPR23}
B. Geshkovski, C. Letrouit, Y. Polyanskiy, P. Rigollet.
A mathematical perspective on Transformers. 
\textit{arXiv preprint}, arXiv:2312.10794 (2023).

\bibitem{H80} J. Hale.
{\it Ordinary Differential Equations.}
Krieger Publishing Company (1980).

\bibitem{Haus99}
J. Henrikson. 
Completeness and total boundedness of the Hausdorff metric. 
\textit{MIT Undergraduate Journal of Mathematics}, 1: 69-80 (1999).


\bibitem{LK06}
J.-S Li, N. Khaneja.
Control of inhomogeneous quantum ensembles,
{\it Phys. Rev. A}, 73(3) (2006).
doi: 10.1103/PhysRevA.73.030302

\bibitem{LoZu16}
J. Loh\'eac J, E. Zuazua.
From averaged to simultaneous controllability of parameter dependent finite-dimensional systems. 
\textit{Annales de la Facult\'e des Sciences de Toulouse: Math\'ematiques}, 25(4):785–828 (2016).
doi: 10.5802/afst.1511

\bibitem{MP18}
R. Murray, M. Palladino.
A model for system uncertainty in
reinforcement learning.
\textit{Syst. Control Lett.},
 122:24-31 (2018).
 doi: 10.1016/j.sysconle.2018.09.011


\bibitem{PPF21}
A. Pesare, M. Palladino, M. Falcone.
Convergence results for an averaged
LQR problem with applications to Reinforcement Learning.
\textit{ Math. Control Signals Syst.}, 33:379–411 (2021).
doi: 10.1007/s00498-021-00294-y


\bibitem{PRG}
C. Phelps, J.O. Royset, Q. Gong.
Optimal control of uncertain systems using sample average approximations.
{\it SIAM J. Control  Optim.},
54(1): 1-29 (2016). 
doi: 10.1137/140983161





\bibitem{RABS}
R. Robin, N. Augier, U. Boscain, M. Sigalotti. 
Ensemble qubit controllability with a single control via adiabatic and rotating wave approximations. 
\textit{J. Diff. Equations}, 318: 414-442 (2022).
doi: 10.1016/j.jde.2022.02.042

\bibitem{R97}
R.T. Rockafellar.
{\it Convex Analysis.}
{Princeton University Press},
1997.

\bibitem{RuZu23}
D. Ruiz-Balet, E. Zuazua.
Neural ODE Control for Classification, Approximation, and Transport.
\textit{SIAM Review}, 65(3):735-773 (2023).
doi: 10.1137/21M1411433

\bibitem{RuLi12}
J. Ruths, J.-S. Li.
Optimal control of inhomogenous ensembles.
{\it IEEE Trans. Aut. Control},
57(8):2021-2032 (2012).
doi: 10.1109/TAC.2012.2195920

\bibitem{Scag22}
A. Scagliotti.
Deep Learning approximations of diffeomorphisms via linear-control systems.
\textit{Math. Control Rel. Fields}, 
13(3): 1226-1257 (2023).
doi: 10.3934/mcrf.2022036

\bibitem{Scag23}
A. Scagliotti.
Optimal control of ensembles of dynamical systems.
\textit{ESAIM: Control Optim. Calc. Var.}, 29 (2023).
doi: 10.1051/cocv/2023011
 
 
\bibitem{Schoe23}
M. Sch\"onlein.
Polynomial methods to construct inputs for uniformly ensemble reachable linear systems. 
\textit{Math. Control Signals Syst.}, 36:251-296 (2024). 
doi: 10.1007/s00498-023-00364-3
 

\bibitem{Shor98}
N.Z. Shor. 
\textit{Nondifferentiable Optimization and Polynomial Problems}.
Springer, Dordrecht (1998).
doi: 10.1007/978-1-4757-6015-6

\bibitem{SS80}
Y. Sakawa, Y. Shindo.
{On global convergence of an algorithm for optimal 
control.}
{\it IEEE Trans. Automat. Contr.},
25(6): 1149-1153 (1980).


\bibitem{V05}
R.B. Vinter.
{Minimax Optimal Control.}
\textit{SIAM J. Control Optim.},
44(3): 939-968 (2005).
doi: 10.1137/S0363012902415244 

\bibitem{Vis84}
A. Visintin.
Strong convergence results related to strict convexity.
\textit{Commun. Partial Differ. Equ.},
9(5): 439-466 (1984).
doi: 10.1080/03605308408820337

\end{thebibliography}
\end{document}